\documentclass[11pt]{amsart}

\usepackage{times}
\usepackage{mathtools}
\usepackage[margin=1in]{geometry}
\usepackage{amssymb}
\usepackage{amsmath}
\usepackage{amsthm}
\usepackage{dsfont}
\usepackage{enumerate}
\usepackage{graphicx}
\usepackage{tikz}
\usepackage{tikz-cd}
\usepackage{semmac}

\theoremstyle{plain}
\newtheorem{theorem}{Theorem}[section]
\newtheorem{lemma}[theorem]{Lemma}
\newtheorem{lemdef}[theorem]{Lemma--Definition}
\newtheorem{proposition}[theorem]{Proposition}
\newtheorem{corollary}[theorem]{Corollary}

\newtheorem{conj}[theorem]{Conjecture}

\theoremstyle{definition}
\newtheorem{definition}[theorem]{Definition}
\newtheorem{fact}[theorem]{Fact}
\newtheorem{remark}[theorem]{Remark}
\newtheorem*{notation}{Notation}
\newtheorem*{ack}{Acknowledgements}
\newtheorem{example}[theorem]{Example}

\title{Explicit root numbers of abelian varieties}
\author{Matthew Bisatt}
\address{Howard House, University of Bristol, Bristol, BS8 1SD, UK}
\email{matthew.bisatt@bristol.ac.uk}

\date{\today}

\begin{document}
\global\long\def\fp{\mathbb{F}_p}
\def\Q{\mathbb{Q}}
\global\long\def\qp{\mathbb{Q}_p}
\global\long\def\R{\mathbb{R}}
\global\long\def\C{\mathbb{C}}
\global\long\def\Z{\mathbb{Z}}
\global\long\def\N{\mathbb{N}}
\global\long\def\FF{\mathcal{F}}
\global\long\def\KK{\mathcal{K}}
\global\long\def\Kx{K^{\times}}
\global\long\def\ss{\mathfrak{s}}
\global\long\def\xcyc{\chi_{cyc}}
\global\long\def\sign{\operatorname{sign}}
\global\long\def\rk{\operatorname{rk}}
\global\long\def\Gal{\operatorname{Gal}}
\global\long\def\Jac{\operatorname{Jac}}
\global\long\def\sp{\operatorname{sp}}
\global\long\def\Hom{\operatorname{Hom}}
\global\long\def\Ind{\operatorname{Ind}}
\global\long\def\Res{\operatorname{Res}}
\global\long\def\Inf{\operatorname{Inf}}
\global\long\def\Stab{\operatorname{Stab}}
\global\long\def\ord{\operatorname{ord}}
\global\long\def\ann{\operatorname{Ann}}
\global\long\def\Tr{\operatorname{Tr}}
\global\long\def\Frobp{\operatorname{Frob}_p}
\global\long\def\Frob{\operatorname{Frob}}
\global\long\def\GL{\operatorname{GL}}
\global\long\def\Aut{\operatorname{Aut}}
\global\long\def\End{\operatorname{End}}
\global\long\def\Id{\operatorname{Id}}

\begin{abstract}
	The Birch and Swinnerton-Dyer conjecture predicts that the parity of the algebraic rank of an abelian variety over a global field should be controlled by the expected sign of the functional equation of its $L$-function, known as the global root number. In this paper, we give explicit formulae for the local root numbers as a product of Jacobi symbols. This enables one to compute the global root number, generalising work of Rohrlich who studies the case of elliptic curves. We provide similar formulae for the root numbers after twisting the abelian variety by a self-dual Artin representation. As an application, we find a rational genus two hyperelliptic curve with a simple Jacobian whose root number is invariant under quadratic twist.
\end{abstract}

\maketitle
\tableofcontents


\section{Introduction}
The primary arithmetic invariant of an abelian variety $A$ defined over a global field $\KK$ is the rank of its Mordell--Weil group; the Birch and Swinnerton-Dyer conjecture predicts that this should be equal to order of vanishing of the (conjectural) $L$-function attached to $A$ at $s=1$. The associated functional equation of the $L$-function gives rise to the parity conjecture which connects the rank to the global root number $W(A/\KK)=\pm 1$, the expected sign of the functional equation.

\begin{conj}[Parity Conjecture]
	$(-1)^{\rk A/\KK}=W(A/\KK).$
\end{conj}

As the global root number is defined independently of the $L$-function, the parity conjecture may be stated without the assumption that the ``$L$-function" attached to $A/\KK$ is indeed an $L$-function; i.e. that the Hasse--Weil conjecture holds. We note that the parity conjecture has been proved in the case of elliptic curves over number fields under some finiteness assumption on the Tate--Shafarevich group in \cite{DD11} by utilising the explicit formulae for root numbers.

Our main objective is to give explicit formulae for the global root number of abelian varieties. Building on the work of Sabitova \cite{Sab07}, we also give explicit formulae for the root number of an abelian variety twisted by a self-dual representation. 

Our reason for considering the twisted root number is due to a generalisation of the parity conjecture. For a finite Galois extension $\FF/\KK$, let $A(\FF)_{\C}$ denote the $\Gal(\FF/\KK)$-representation on $A(\FF) \otimes_{\Z} \C$ and let $\tau$ be an Artin representation of $\Gal(\FF/\KK)$ with real trace. As a consequence of the Birch and Swinnerton-Dyer conjecture for Artin twists \cite[p.127]{Roh90}, we get the following conjecture for the twisted root number.

\begin{conj}[Twisted Parity Conjecture]
	$(-1)^{\langle A(\FF)_{\C}, \tau \rangle} = W(A/\KK,\tau).$
\end{conj}


As an example we use our formulae to compute the global root number of the Jacobian of a hyperelliptic curve in \S \ref{example}, using results of \cite{DDMM}. In \S \ref{quad}, we give an application of our results by deriving sufficient criteria for abelian varieties with the property that the parity of their rank is invariant under quadratic twist (assuming the parity conjecture); a table of hyperelliptic curves whose Jacobians have this property is given in Appendix \ref{lawfultable}. In the setting of elliptic curves, Mazur and Rubin conjecture \cite[Conjecture 1.3]{MR10} that this is one of only two phenomena\footnote{The other being the existence of rational $2$-torsion points.} that prevent an elliptic curve from having quadratic twists of any given $2$-Selmer rank. In particular, this would imply that every elliptic curve over a number field $\KK$ has a quadratic twist of rank at most $2$; we might expect a similar statement for abelian varieties.

This work follows a recent pattern of extending explicit arithmetic of elliptic curves to abelian varieties. Other such work includes the work of Booker, Sijsling, Sutherland, Voight and Yasaki \cite{BSSVY16} who have computed various invariants of genus two hyperelliptic curves as part of the LMFDB collaboration \cite{LMFDB}; van Bommel's algorithms \cite{vB17} to numerically verify the full Birch and Swinnerton-Dyer conjecture for Jacobians of hyperelliptic curves; and the upcoming work of Dokchitser, Dokchitser, Maistret and Morgan \cite{DDMM} who study the arithmetic of hyperelliptic curves. Moreover, we can recover the root number computations done in the recent paper of Brumer, Kramer and Sabitova \cite{BKS18} (cf. Remark \ref{BKS}).

For reference, we give Rohrlich's result for root numbers of elliptic curves in terms of Jacobi symbols. This is the result we will generalise to abelian varieties.

\begin{theorem}\cite[Theorem 2]{Roh96}
\label{Rohr}
	Let $K$ be a finite extension of $\qp$ with normalised valuation $v$ and residue field of cardinality $q$. Let $E/K$ be an elliptic curve and let $j,\Delta$ denote the $j$-invariant and discriminant of $E$ respectively.
\begin{enumerate}
\item If $v(j)<0$ and $p \geqslant 3$, then 

$W(E/K)=\begin{cases}
-1 &\text{ if $E/K$ has split multiplicative reduction}; \\
\quad 1 &\text{ if $E/K$ has nonsplit multiplicative reduction}; \\
\left( \dfrac{-1}{q} \right) &\text{ otherwise.} \end{cases}$
\item If $v(j)\geqslant 0$ and $p \geqslant 5$, set $e=\dfrac{12}{\gcd(v(\Delta),12)}$. Then

$W(E/K)=\begin{cases}
\quad 1 &\text{ if } e=1; \\
\left( \dfrac{-1}{q} \right) &\text{ if } e=2 \quad \text{ or } e=6; \\[10pt]
\left( \dfrac{-3}{q} \right) &\text{ if } e=3; \\[10pt]
\left( \dfrac{-2}{q} \right) &\text{ if } e=4. \\[10pt]
\end{cases}$
\end{enumerate}
\end{theorem}

These formulae have since been used to address several questions including the density of rational points on elliptic surfaces \cite{BDD16,Des16,VA11}, the behaviour of root numbers in families of elliptic curves \cite{Hel04} and to prove the $2$-parity conjecture \cite{DD11}. Our generalisation of Rohrlich's formulae sets the stage for answering these questions in higher dimensions.

\subsection{Notation}
		
We now set up some notation that we frequently use throughout the paper. For definitions of the root number of a representation, reduction types of an abelian variety and $\rho_A$ see \S \ref{background}.\\

\noindent \begin{tabular}{l p{0.8\linewidth}}
$K, K^{\rm{sep}}$ &  non-Archimedean local field with separable closure $K^{\rm{sep}}$, \\
$\mathcal{O}_K, \pi_K$ & ring of integers of $K$ with uniformiser $\pi_K$, \\
$v:K^{\times} \twoheadrightarrow \Z$ & normalised valuation of $K$, \\
$p>0$ & residue characteristic of $K$, \\
$q$ & cardinality of the residue field of $K$, \\
$I$ & the absolute inertia group of $K$, \\
$\Frob$ & an arithmetic Frobenius element of $K$, \\
$\mathcal{W}(K^{\rm{sep}}/K)$ & $\cong I \rtimes \langle \Frob \rangle,$ the absolute Weil group of $K$, \\
$\xcyc$ & the cyclotomic character, i.e. the unique unramified character of  $\mathcal{W}(K^{\rm{sep}}/K)$ such that $\xcyc(\Frob)=q$, \\
$A/K$ & an abelian variety, \\
$\rho_A$ & $\cong \rho_B \oplus (\rho_T \otimes \xcyc^{-1} \otimes \sp(2))$ the complexification of the canonical $\ell$-adic representation of $A/K$, where $\rho_B,\rho_T$ have finite image of inertia (cf. Fact \ref{decomp}), \\
$W(A/K)$ & $=W(\rho_A)$ the local root number of $A/K$. \\
\end{tabular}

\vspace{10pt}
	Unless otherwise specified, we shall suppose throughout that $\rho_A$ is tamely ramified, i.e. the image of wild inertia is trivial; this is always true if $p>2\dim A+1$ \cite[p.497]{ST68}. We therefore introduce some further notation for this setting. \\

\noindent \begin{tabular}{l p{0.8\linewidth}}
$\iota \in I$ & any lift of a topological generator of the tame inertia group, \\
$\tilde{\varphi}(e)$ & $=\begin{cases}
		 2 &\text{ if } e=1,2, \\
		 |\left(\Z/e\Z\right)^{\times}| &\text{ if } e \geqslant 3,
	\end{cases}$ \\
$m_e \in \Z$ & $= |\{ \text{eigenvalues of } \rho_B(\iota) \text{ of order } e \}| / \tilde{\varphi}(e)$ (counting multiplicity) for each positive integer $e$, \\		
$m_T$ & multiplicity of $-1$ as an eigenvalue of $\rho_T(\iota)$, \\
$\left( \dfrac{\ast}{\ast}\right)$ & the Jacobi symbol, \\
$\zeta_n$ & a primitive $n^{\rm{th}}$ root of unity, \\
$\rho_{e,f}$ & the direct sum of all faithful irreducible $f$-dimensional representations of $\Gal(K(\zeta_e,\pi_K^{1/e})/K)$, where $f=[K(\zeta_e):K]$, \\
$\langle \, , \, \rangle$ & the standard inner product of characters of representations, embedding into $\C$ if necessary, \\
$\rho^*$ & the contragredient or dual of a representation $\rho$, \\
$\tau_v$ & a complex Artin representation of $\Gal(K^{\rm{sep}}/K)$ which is self-dual, i.e. $\tau_v^* \cong \tau_v$, \\
$W(A/K,\tau_v)$ & $=W(\rho_A \otimes \tau_v)$ the twisted (local) root number.
\end{tabular}

\begin{remark}
	Note that the characteristic polynomial of $\rho_B(\iota)$ is defined over $\Z$ since its roots are integral and independent of $\ell\neq p$ \cite[Th\'{e}or\`{e}me 4.3b]{Gro72}. This implies that it is a product of cyclotomic polynomials and hence $m_e \in \Z$ for $e\geqslant 3$. Moreover, since $\rho_B(\iota)$ has trivial determinant and even dimension, the same holds for $m_1$ and $m_2$; see Remark \ref{eulertweak} for an explanation of why $\tilde{\varphi}(e)$ is not the usual Euler totient function in these cases.
\end{remark}

\subsection{Statement of results}	
Here we detail the main results of the paper. We concern ourselves with the root numbers of both the abelian variety $A/K$ and its twist by $\tau_v$. We are fortunate to once again be able to write the local root number compactly as a product of Jacobi symbols. We further consider the representation $\rho_T$ as a (self-dual) complex Artin representation and hence will be semisimple.

\begin{theorem}[=Theorem \ref{rootno}]
	Let $A/K$ be an abelian variety over a non-Archimedean local field which has tame reduction. Then $$W(A/K)=\left(\prod\limits_{e \in \N} W_{q,e}^{m_e} \right) (-1)^{\langle \mathds{1}, \rho_T \rangle} W_{q,2}^{m_T},$$ where:
		
	$W_{q,e} = \begin{cases}
		\left( \dfrac{q}{l} \right) &\text{ if } e=l^k \quad \text{ for some odd prime $l$ and integer $k\geqslant 1$}; \\[10pt]	
		\left( \dfrac{-1}{q} \right) &\text{ if } e=2l^k \quad \text{ for some prime } l \equiv 3 \bmod{4}, k\geqslant 1,  \qquad \text{ or } e=2; \\[10pt]	
		\left( \dfrac{-2}{q} \right) &\text{ if } e=4; \\[10pt]	
		\left( \dfrac{2}{q} \right) &\text{ if } e=2^k \quad \text{ for some integer } k \geqslant 3; \\[10pt]	
		\quad 1 &\text{ else.}
	\end{cases}$
\end{theorem}

\begin{remark}
	When $A=E$ is an elliptic curve, we recover the result of Rohrlich (Theorem \ref{Rohr}). Indeed, if $E/K$ has potentially good reduction $(v(j)\geqslant 0)$ then $\rho_T=0$ and the two eigenvalues of $\rho_E(\iota)$ have order $e=\frac{12}{\gcd(v(\Delta_E),12)}$ (see \cite[Proposition 2v]{Roh93}) and our corresponding Jacobi symbols then match. Now suppose that $E/K$ has potentially multiplicative reduction, so $\rho_B=0$ and $v(j)<0$. Then using the theory of Tate curves, one can show that $\rho_T=\chi$ is the character corresponding to the minimal extension where $E/K$ attains split multiplicative reduction. More explicitly, $\chi$ is unramified if and only if $E/K$ has multiplicative reduction and $\chi=\mathds{1}$ if it is split over $K$ \cite[Proposition 3]{Roh93}; this then coincides with our result.
\end{remark}	
	
We shall now give the result for the twisted root numbers. We briefly mention that we identify the determinant character $\det \tau_v$ with a character of $K^{\times}$ under the Artin map.

\begin{theorem}[=Theorem \ref{twistthm}]
	Let $A/K$ be an abelian variety over a non-Archimedean local field which has tame reduction and let $\tau_v$ be a self-dual Artin representation of $\Gal(K^{\rm{sep}}/K)$. Then $$W(A/K, \tau_v)=W(A/K)^{\dim \tau_v}((\det \tau_v)(-1))^{\dim A}(-1)^{l_1+l_2},$$ where $$l_1=\langle \rho_T,\tau_v \rangle + \dim \tau_v \langle \mathds{1}, \rho_T \rangle,$$ $$l_2= \sum\limits_{e \in \N} m_e \left( \langle \rho_{e,f},\tau_v \rangle + \dfrac{\tilde{\varphi}(e)}{[K(\zeta_e):K]}(\langle \mathds{1},\tau_v \rangle + \langle \eta, \tau_v \rangle + \dim \tau_v)\right),$$ with $\eta$ the unramified quadratic character of $K$.
\end{theorem}
	
\begin{remark}
\label{infty}
	For completeness, we give the result for the Archimedean places (see \cite[Lemma 2.1]{Sab07}) since we shall use it in the global case. Let $F=\R$ or $\C$ and let $\tau_v$ be an Artin representation of $\Gal(F^{\rm{sep}}/F)$. Then $$W(A/F, \tau_v)= (-1)^{(\dim A)(\dim \tau_v)}.$$
\end{remark}
	
We have also collated this into a formula for the global twisted root number. Let $\KK$ be a global field, and for each place $v$, we extend it to $\KK^{\rm{sep}}$ and take $\tau_v$ to be the restriction of a self-dual Artin representation $\tau$ of $\Gal(\KK^{\rm{sep}}/\KK)$ to the decompostion group at $v$.
	
Moreover, we define $m_{e,v}$ to be equal to $m_e$ for the abelian variety $A/\KK_v$ and note that $\rho_{e,f}$ depends on $v$.
	
\begin{theorem}[=Theorem \ref{globtwist}]
	Let $\KK$ be a global field, $A/\KK$ an abelian variety and $\tau$ a finite dimensional Artin representation with real character. Let $M_{\KK}$ be the set of places of $\KK$. For each finite place $v \in M_{\KK}$, write $\rho_{A/\KK_v}=\rho_{B_v} \oplus (\rho_{T_v} \otimes \xcyc^{-1} \otimes \sp(2))$ where $\rho_{B_v},\rho_{T_v}$ have finite image of inertia. If $\tau_v$ is ramified, assume ${A/\KK_v}$ has tame reduction. Then
	\begin{equation*}
		W(A/\KK, \tau) = W(A/\KK)^{\dim \tau} (\sign (\det \tau))^{\dim A} \cdot T \cdot S,
	\end{equation*}
	where
	\begin{eqnarray*}
	\sign(\det \tau) &=& \prod\limits_{v| \infty, \, v \in M_{\KK}} (\det \tau_v)(-1), \\
	T &=& \prod\limits_{v<\infty, v \in M_{\KK}}(-1)^{\langle \rho_{T_v},\tau_v \rangle + \dim \tau \langle \mathds{1}, \rho_{T_v} \rangle}, \\
	S &=& \prod\limits_{v<\infty, v \in M_{\KK}} \prod\limits_{e \in \N} \left( (-1)^{\langle \rho_{e,f}, \tau_v \rangle + \frac{\tilde{\varphi}(e)}{[\KK_v(\zeta_e):\KK_v]}(\langle \mathds{1},\tau_v \rangle + \langle \eta_v, \tau_v \rangle + \dim \tau)}\right)^{m_{e,v}},
	\end{eqnarray*}
	and $\eta_v$ is the unramified quadratic character of $\KK_v^{\times}$.
\end{theorem}

As an application of these results, we are also able to provide sufficient criteria for an abelian variety over a global field whose global root number is invariant under quadratic twist. In particular if we can find such an abelian variety with odd rank, then all of its quadratic twists should have infinitely many rational points by the parity conjecture. \\

\noindent \textbf{Criterion A.}
	Let $A/K$ be an abelian variety over a non-Archimedean local field. Suppose that $A/K$ has tame reduction and that $p$ is odd or $\rho_T$ is zero.

	Let $$W_g=\prod\limits_{2 \nmid e} W_{q,e}^{m_e} \prod\limits_{e=4 \, or \, 2||e} W_{q,e/2}^{m_e},$$ where $2||e$ means that $v_2(e)=1$, i.e. $e \equiv 2 \bmod{4}$.
	
	If $p$ is odd, we moreover let $\eta$ be the unramified quadratic character of $K$ and let $\chi_1, \chi_2$ be the ramified quadratic characters of $K$. We define $$\langle \rho_{T}, \mathds{1} \rangle = n_1, \qquad \langle \rho_{T}, \eta \rangle = n_2, \qquad \langle \rho_{T}, \chi_1 \rangle = n_3, \qquad \langle \rho_{T}, \chi_2 \rangle = n_4.$$
	
	Then $A/K$ satisfies Criterion A if any of the following conditions hold: 
	\begin{enumerate}
	\item $p=2$ and $W_g=1$;
	\item $p$ is odd, $n_1 \equiv n_2 \mod{2},$ $n_3 \equiv n_4 \mod{2},$ and $W_g=(-1)^{n_1+n_3}$.
	\end{enumerate}

\begin{theorem}[=Lemma \ref{inflaw} and Theorem \ref{lawful}]
	Let $A/\KK$ be an abelian variety over a global field. Then the global root number of \emph{every} quadratic twist of $A/\KK$ is equal to $W(A/\KK)$ if both of the following criteria are satisfied:
	\begin{enumerate}
	 \item $\dim A$ is even or $\KK$ has no real places;
	 \item for every finite place $v$, $A/\KK_v$ satisfies Criterion A.
	\end{enumerate}
\end{theorem}

\begin{ack}
	I would like to thank Vladimir Dokchitser for many helpful discussions and the referee for useful comments. I also extend my gratitude to the University of Warwick and King's College London where this research was done.
\end{ack}

\section{Background on root numbers and the $\ell$-adic representation}
\label{background}	
Local root numbers are defined in terms of $\varepsilon$-factors; in this section we shall briefly state the important properties we use. For more detail, see for example \cite[pp.13,15]{Tat79} or \cite[p.144]{Roh94}.
	
Let $K$ be a local field, $\chi: K^{\times} \rightarrow \C^{\times}$ a character. Using the Artin map, we can identify $\chi$ with a character of the Weil group. Let $\psi$ be a non-trivial additive character of $K$ and $dx$ an additive Haar measure on $K$. Tate gives explicit formulae for computing $\varepsilon(\chi,\psi,dx) \in \C^{\times}$ and discusses the properties of $\varepsilon$-factors that allow one to uniquely extend the definition to arbitrary dimension.
	
\begin{theorem}
\label{props}
	Let $\chi$ be a one dimensional unramified complex representation of the Weil group over $K$. Then $$\varepsilon(\chi,\psi,dx)=\dfrac{\chi(\pi_K^{n(\psi)})}{||\pi_K^{n(\psi)}||} \int_{\mathcal{O}_K} dx,$$ where $\pi_K$ is a uniformiser and $n(\psi)$ is the conductor exponent of $\psi$, i.e. the largest $n$ such that $\psi(\pi^{-n}\mathcal{O}_K)=1$.
		
	Let $\rho, \rho'$ be two finite dimensional Weil representations of $\mathcal{W}(K^{\rm{sep}}/K)$ and let $a(\rho)$ denote the Artin conductor exponent of $\rho$.
	\begin{enumerate}
	\item $\varepsilon(\rho \oplus \rho',\psi,dx)=\varepsilon(\rho,\psi,dx)\varepsilon(\rho',\psi,dx)$.
	\item Let $L/K$ be a finite extension. Let $\sigma, \sigma'$ be finite dimensional Weil representations of $\mathcal{W}(K^{\rm{sep}}/L)$ such that $\dim \sigma = \dim \sigma'$ and fix an additive Haar measure $dx_L$ on $L$. Then $$\dfrac{\varepsilon(\Ind_{L/K} \sigma,\psi, dx_L)}{\varepsilon(\Ind_{L/K} \sigma',\psi, dx_L)}=\dfrac{\varepsilon(\sigma,\psi \circ \Tr_{L/K},dx)}{\varepsilon(\sigma',\psi \circ \Tr_{L/K},dx)},$$ where $\Tr_{L/K}$ is the trace map and $\Ind_{L/K}$ is induction from $\mathcal{W}(K^{\rm{sep}}/L)$ to $\mathcal{W}(K^{\rm{sep}}/K)$.
	\item $\varepsilon(\rho \otimes \rho', \psi, dx)=\varepsilon(\rho,\psi, dx)^{\dim \rho'} ((\det \rho')(\pi_K^{a(\rho)+n(\psi)\dim \rho}))$ if $\rho'$ is unramified.
	\item $\varepsilon(\rho \oplus \rho^*, \psi, dx)= ((\det\rho)(-1))q^{n(\psi)\dim(\rho)+a(\rho)}$.
	\end{enumerate}
\end{theorem}

\begin{remark}
	One can also define $\varepsilon$-factors for a Weil--Deligne representation by connecting it to the $\varepsilon$-factor of its semisimplification; we will however use known results for our computations in this case so will not concern ourselves with this.
\end{remark}
	
\begin{definition}
	Let $\rho, \psi, dx$ be as above. Then the local root number of $\rho$ is $$W(\rho,\psi,dx)=\dfrac{\varepsilon (\rho,\psi,dx)}{|\varepsilon (\rho,\psi,dx)|}.$$
\end{definition}
	
\begin{corollary}(to Theorem \ref{props})
\label{dualrep}
	Let $\rho$ be a finite dimensional Weil representation, $\psi$ an additive character and $dx$ a Haar measure. Then:
	\begin{enumerate}
		\item $W(\rho \otimes \xcyc^k, \psi, dx)=W(\rho, \psi, dx)$ for any $k \in \R$;
		\item $W(\rho \oplus \rho^*, \psi, dx)= (\det \rho)(-1)$.
	\end{enumerate}
\end{corollary}
	
\begin{definition}
	Let $A/K$ be an abelian variety over a local field. Then the canonical Weil--Deligne representation $\rho_A$ is the complexification of the $\ell$-adic Galois representation acting on $$(\varprojlim\limits_n A[\ell^n] \otimes_{\Z_{\ell}} \Q_\ell)^* \cong H^1_{\acute{e}t}(A/K^{\rm{sep}},\Q_\ell),$$ for any rational prime $\ell$ different to the residue characteristic of $K$.\footnote{This is independent of the choice of $\ell$ \cite[Th\'{e}or\`{e}me 4.3b]{Gro72}.}
\end{definition}

\begin{fact}\cite[Proposition 1.10]{Sab07}
\label{decomp}
	Let $A/K$ be an abelian variety over a local field. Then there exists a Galois representation $\rho_T$ and a semisimple Weil--Deligne representation $\rho_B$ such that:
	\begin{enumerate}
		\item $\rho_B$ has finite image of inertia;
		\item $\rho_B \otimes \xcyc^{1/2}$ is symplectic;
		\item $\rho_T: \Gal(K^{\rm{sep}}/K) \rightarrow \GL_r(\Z)$ for some $0 \leqslant r \leqslant \dim A$;
		\item $\rho_A \cong \rho_B \oplus (\rho_T \otimes \xcyc^{-1} \otimes \sp(2))$,
	\end{enumerate}
	where $\xcyc$ and $\sp(2)$ are the cyclotomic character and special representation of dimension $2$ respectively.
\end{fact}

\begin{remark}
	Knowing that $\rho_T$ has image with integral entries is useful for Lemma \ref{mult} later, but otherwise we can embed $\Z \hookrightarrow \C$ and assume that $\rho_T$ is semisimple when we use inner products.
\end{remark}	

\begin{remark}
	In fact, there exists an abelian variety $B/K$ with potentially good reduction whose associated Weil--Deligne representation is isomorphic to $\rho_B$.
\end{remark}

\begin{definition}
	Let $A/K$ be an abelian variety and decompose $\rho_A$ as in Fact \ref{decomp}. We say $A/K$ has
	\begin{enumerate}
		\item \emph{tame reduction} if $A/L$ is semistable for some finite tamely ramified extension $L/K$;
		\item \emph{potentially good reduction} if $\rho_T$ is the zero representation;
		\item \emph{potentially totally toric reduction} if $\rho_B$ is the zero representation.
	\end{enumerate}
\end{definition}

\begin{definition}
	Let $A/\KK$ be an abelian variety over a global or local field and let $\tau$ be an Artin representation of $\Gal(\KK^{\rm{sep}}/\KK)$. 
	\begin{enumerate}
		\item If $\KK=K$ is a local field, then $$W(A/K, \tau):=W(\rho_A \otimes \tau, \psi,dx).$$
		\item If $\KK$ is a global field, then $$W(A/\KK, \tau):= \prod\limits_{v \in M_\KK} W(A/\KK_v,\tau_v),$$ where $M_{\KK}$ is the set of places of $\KK$ and $\tau_v$ is the restriction of $\tau$ to the decomposition group $\Gal(\KK_v^{\rm{sep}}/\KK_v)$.
	\end{enumerate}
	In both cases, we write $W(A/\KK)$ for $W(A/\KK, \mathds{1})$.
\end{definition}

\begin{remark}		
	It is clear from the definition of $\varepsilon$-factors that the root number is independent of the choice of Haar measure $dx$. It is also independent of $\psi$ whenever $\rho$ is symplectic \cite[p.315]{Roh96} which will always be the case. We have therefore chosen to fix $\psi$ such that $\ker \psi = \mathcal{O}_K$ throughout to further ease notation, although our method can be adapted to prove this independence directly.
\end{remark}

\section{Potentially good reduction}
\subsection{Decomposition of the representation}	
We first study the Weil--Deligne representation $\rho_B$ which can identified with a representation of the Weil group with finite image of inertia.
	
Our assumption that $\rho_A$ (and hence also $\rho_B$) is tamely ramified implies that $p$ will necessarily be coprime to the order of the image of the inertia group. We shall want to deal with the irreducible summands of $\rho_B$ so our first step in that direction is the following simple lemma about $\rho_B(I)$, which follows directly from the structure of the Weil group.

\begin{lemma}
\label{reduce}
	Let $\rho: \mathcal{W}(K^{\rm{sep}}/K) \rightarrow \GL(V)$ be a tamely ramified representation with finite image of inertia. Suppose the characteristic polynomial of $\rho(\iota)$ has coefficients in $\Z$. If the characteristic polynomial is reducible over $\Z$, then $\rho$ is also reducible as a Weil representation.
\end{lemma}
	
Therefore by decomposing $\rho_B$ if necessary, we may assume that the eigenvalues of $\rho_B(\iota)$ are all primitive $e^{\rm{th}}$ roots of unity for some fixed $e$. The irreducible summands of $\rho_B$ are one-dimensional unramified twists of representations of Galois type.
	
\begin{definition}
	Let $\rho: \mathcal{W}(K^{\rm{sep}}/K) \rightarrow \GL(V)$ be a Weil representation. Then $\rho$ is said to be of Galois type if it factors through an open subgroup of finite index. Equivalently, there exists a finite Galois extension $L/K$ such that $\rho$ is trivial on $\mathcal{W}(K^{\rm{sep}}/L)$.
\end{definition}

Note $\dfrac{\mathcal{W}(K^{\rm{sep}}/K)}{\mathcal{W}(K^{\rm{sep}}/L)} \cong \Gal(L/K)$ so we may (and freely do so) identify representations of Galois type with Artin representations. We therefore study Artin representations which factor through a finite Galois extension $L/K$. Since we assume that $\rho_B$ is tamely ramified, we shall only concern ourselves with the case when $L/K$ is tamely ramified. We may hence assume that our Artin representations factor (not necessarily faithfully) through a tame Galois extension $L/K$ with $$\Gal(L/K)=\langle \iota, \Frob | \iota^e, \Frob^n, \Frob \iota \Frob^{-1} = \iota^q \rangle \cong C_e \rtimes C_n,$$ where $e,n$ are the ramification and residue degrees of $L/K$ respectively. Note that the order $f$ of $q \bmod{e}$ necessarily divides $n$. Moreover, we may suppose that the Artin representation $\theta$ is faithful on the inertia subgroup by factoring through the kernel of $\theta$ restricted to inertia if necessary.
	
\begin{notation} 
Throughout the remainder of this section, we let $\theta$ denote an irreducible, Artin representation of $\Gal(L/K)=\langle \iota, \Frob| \iota^e, \Frob^n, \Frob \iota \Frob^{-1} = \iota^q \rangle$, where, $p \nmid e$, $\Aut_{\langle \Frob \rangle}(\langle \iota \rangle) \cong C_f$ (i.e. $q \bmod{e}$ has order $f$) and we suppose that $f>1$ and hence $e>2$.
\end{notation}
	
\begin{lemma}
\label{induce}
Let $\theta$ be an irreducible, tamely ramified representation of $\Gal(L/K)$ which is faithful on the inertia subgroup. Then:
	\begin{enumerate}
	\item $\dim \theta=f$;
	\item $\theta= \Ind_{\langle \iota, \Frob^f \rangle}^{\Gal(L/K)} \chi \otimes \gamma$ where $\chi$ is a character of $\langle \iota \rangle$, $\gamma$ is a character of $\langle \Frob^f \rangle$;
	\item $(\det \theta) (\iota)=\chi\left( \iota^{\frac{q^f-1}{q-1}} \right), \quad
			(\det \theta) (\Frob)=(-1)^{1+\dim \theta}\gamma(\Frob^f);$
	\item $\theta$ is self-dual if and only if $f$ is even, $q^{f/2} \equiv -1 \bmod{e}$ and $\gamma^2=\mathds{1}$; moreover, $\theta$ factors faithfully through $\Gal(L/K)/\ker \gamma$;
	\item if $\theta$ is self-dual, then $\theta$ is orthogonal if and only if $\gamma=\mathds{1}$;
	\item the Artin conductor exponent of $\theta \otimes \nu$ is equal to $f$ for any one dimensional unramified character $\nu$ of $\mathcal{W}(K^{\rm{sep}}/K)$.
	\end{enumerate}	
\end{lemma}
	
\begin{proof}
	The first four statements are purely representation theoretic results and follow from \cite[p.62]{Ser77} (which classifies the irreducible representations of a group which is a semidirect product by an abelian group) and explicitly computing the induced representation. For $(v)$, note that if $\theta$ is symplectic then it has trivial determinant so we necessarily have $\gamma \neq \mathds{1}$; the converse is the content of Proposition \ref{Sab0}. For the final statement, note that there is no inertia invariant subspace.
\end{proof}
			
\begin{lemma}
\label{char}
	Let $\rho$ be a symplectic Weil representation which has finite image of inertia and let $\pi$ be a $1$-dimensional summand of $\rho$. Then there exists another distinct $1$-dimensional summand $\tilde{\pi}$ of $\rho$ such that $\tilde{\pi} \cong \pi^*$.
\end{lemma}
	
\begin{proof}
	This follows from the symplectic pairing.
\end{proof}
	
	We now separate into two cases dependent on whether our irreducible Weil representation is self-dual or not; our method for computing the root number is different in each case.
		
\subsection{Dual pairs}		
	We now consider the irreducible representations	which are not self-dual. Our approach is no different if we consider Weil representations $\rho$ instead of Artin representations $\theta$. We first show that in our context, such representations must arise in pairs and then compute the root number with Corollary \ref{dualrep}.
	
\begin{lemma}
	Let $\rho$ be a self-dual Weil representation and let $\sigma$ be an irreducible summand of $\rho$ which is not self-dual. Then $\sigma^*$ is a distinct irreducible summand of $\rho$.
\end{lemma}
	
\begin{remark}
\label{eulertweak}
	Lemma \ref{char} tells us that this also holds for self-dual characters (which have order dividing 2 on inertia) so we shall deal with all characters under this case and assume $\dim \rho \geqslant 2$ in the self-dual setting. This is the reason for our definition of $\tilde{\varphi}$; such self-dual characters will appear with even multiplicity so our tweak of the usual Euler totient function takes this into account.
\end{remark}

\begin{remark}
	If $A/K$ has good reduction, then $\rho_A=\rho_B$ splits as a sum of unramified characters and hence $W(A/K)=1$.
\end{remark}	
	
	Before we examine $\det \sigma$ in more detail, we shall first turn our attention to the Artin map. This enables us to identify $\det \sigma$ with a character $\phi$ of $\Kx$, where the inertia group corresponds to the unit group $\mathcal{O}_K^{\times}$. Since the image of inertia is finite, we are able to discuss the (necessarily finite) order of $\phi$ on inertia. In addition, $\phi$ is at most tamely ramified; it may actually be unramified (and therefore has order 1) but the criterion we shall shortly give will take this into account. 
	
	The subgroup of wild inertia is identified with the group of principal units and hence we can factor our character through this to get a character on the torsion subgroup consisting of elements with order coprime to $p$, which is isomorphic to the residue field. It is this new character that we now deal with; we shall freely swap between the two interpretations.
		
\begin{lemma}
\label{2reduce}
	Let $q \equiv 1 \mod{r}$ and let $\phi: \Kx \rightarrow \C^{\times}$ be a tamely ramified character such that $\phi(\mathcal{O}_K^{\times})$ is cyclic of order $r$. If $q$ is even, then $\phi(-1)=1$. Otherwise $q$ is odd and $$\phi(-1)=1 \Leftrightarrow q \equiv 1 \bmod{2r}.$$
\end{lemma}
	
\begin{proof}
	First note that if $\phi$ has odd order $r$ (this is true if $q$ is even), then we necessarily have $\phi(-1)=1$, so we may reduce to the case when $\phi$ has $2$-power order. Now as $\phi$ has order $r$, its kernel is the set of $r^{\rm{th}}$ powers and so we are left to determine whether $-1$ is an $r^{\rm{th}}$ power; this gives rise to the stated congruence condition.
\end{proof}
	
	We shall now consider $\det \sigma$ explicitly to get a closed form expression for the corresponding root number. From Lemma \ref{induce}, $(\det \sigma)(\iota)=\zeta_e^{\frac{q^f-1}{q-1}}$ where $\zeta_e$ is a primitive $e^{\rm{th}}$ root of unity, since $\chi$ is faithful. This implies that $(\det \sigma)(\mathcal{O}_K^{\times})$ is cyclic of order $r=\dfrac{e}{\gcd(e,\frac{q^f-1}{q-1})}$. 
	
	The following lemma ascertains that the congruence hypothesis on $q$ in Lemma \ref{2reduce} is satisfied. The proof is straightforward and we therefore choose to omit it here.
	
\begin{lemma}
\label{dualmod}
	Let $e>0$, $q$ a power of an odd prime, $f$ the least positive integer such that $q^f \equiv 1 \mod{e}$. Then $q \equiv 1 \mod{\dfrac{e}{\gcd(e,\frac{q^f-1}{q-1})}}$.
\end{lemma}

The above two lemmas, in conjunction with Corollary \ref{dualrep}, now provide the proof for the following theorem.

\begin{theorem}
\label{dual}
	Let $\sigma$ be a tamely-ramified irreducible $f$-dimensional Weil representation such that $|\sigma(I)|=e$ is finite. If $q$ is even then $W(\sigma \oplus \sigma^*)=1$. Otherwise $$W(\sigma \oplus \sigma^*)=1 \Leftrightarrow q \equiv 1 \bmod{\dfrac{2e}{\gcd(e,\frac{q^f-1}{q-1})}}.$$
\end{theorem}

\begin{remark}
	If $\dim \sigma =1$, then $f=1$ hence $W(\sigma \oplus \sigma^*)=1 \Leftrightarrow q \equiv 1 \mod{2e}$.
\end{remark}

\begin{remark}
\label{dualodde}	
	Observe that if $e$ is odd, then $W(\sigma \oplus \sigma^*)=1$ independently of whether $q$ is odd or even.
\end{remark}	

\subsection{The self-dual case}

In this section, we study irreducible self-dual Weil representations. Our first step is to show that they may be identified with Artin representations.
	
\begin{lemma}
	Let $\rho$ be an irreducible self-dual Weil representation. Then $\rho$ is of Galois type and hence isomorphic to an Artin representation.
\end{lemma}

\begin{proof}
	Note that $\rho \cong \theta \otimes \nu$ for some Artin representation $\theta$ and unramified character $\nu$ by the classification of Weil representations \cite[Proposition 3.1.3]{Deligne1973}. Using the self-duality and that $\theta$ has finite order, we see that $\nu$ has finite order so are done since $\nu^{2\dim \theta} = (\det \theta)^{-2}$.
\end{proof}

This allows us to directly apply the results of Lemma \ref{induce}. Since we have already dealt with the self-dual characters, we suppose that $\dim \theta = f \geqslant 2$. Recall that $\theta$ has the form $\theta=\Ind_H^G \chi \otimes \gamma$, where $H=\langle \iota, \Frob^f \rangle \leq G=\Gal(L/K)$. As $\theta$ is a monomial representation, we use the inductivity property of $\varepsilon$-factors.

To apply Theorem \ref{props}ii, we need to choose an auxiliary character of $H$ to $\chi \otimes \gamma$; the obvious choice here is the trivial character $\mathds{1}_H$. This has several benefits as we shall see. Firstly, its induction is simply the permutation action on the cosets of $H$. In other words, $\Ind_H^G \mathds{1}_H = \C[G/H]$, but since $G/H$ is an abelian group, the induction must also split as a sum of characters. Moreover, these summands are unramified since their restriction to inertia is the trivial character by Frobenius reciprocity.

\begin{lemma}
	Let $L/K$ be a finite Galois extension with Galois group $G$, $H \leqslant G$ a subgroup containing the inertia subgroup and let $\sign_{[G:H]}$ be the sign of the permutation representation\footnote{This is the trivial representation if $|G/H|$ is odd.} on $G/H$. Then $$W(\Ind_H^G \mathds{1}_H) = W(\sign_{[G:H]}).$$
\end{lemma}

\begin{proof}
	Observe that $\Ind_H^G \mathds{1}_H = \mathds{1}_G^t \oplus \sign_{[G:H]} \oplus \bigoplus_{j=1}^k (\chi_j \oplus \chi_j^*)$ where $\chi_j$ are unramified one-dimensional representations, $t \in \{0,1\}$ is positive if and only if $|G/H|$ is even. Now by Corollary \ref{dualrep}, $W(\chi_j \oplus \chi_j^*)=\chi_j(-1)=1$, and a direct computation shows $W(\mathds{1}_G)=1$.
\end{proof}

The inductivity property (Theorem \ref{props}ii) for root numbers gives us $$
\dfrac{W(\theta)}{W(\Ind_H^G \mathds{1}_H)} = \dfrac{W(\chi \otimes \gamma)}{W(\mathds{1}_H)},$$and so using the above we have reduced this to $W(\theta)=W(\sign_{[G:H]})W(\chi \otimes \gamma)$ since $W(\mathds{1}_H)=1$. Moreover, one computes via Theorem \ref{props} that $W(\sign_{[G:H]})=1$, and hence $W(\theta)=W(\chi \otimes \gamma)$.

All that remains is to compute $W(\chi \otimes \gamma)$, for which we use the theorem of Fr\"{o}hlich and Queyrut \cite[p.130]{FQ73}.

\begin{theorem}[Fr\"{o}hlich-Queyrut]
	Let $K_1$ be a local field, $K_2$ a quadratic extension of $K_1$. Let $u \in K_2$ be such that $K_2=K_1(u)$ and $u^2 \in K_1$. Then if $\lambda$ is a character of $K_2^{\times}$ which is trivial on $K_1^{\times}$, then $W(\lambda)=\lambda(u)$.
\end{theorem}

Observe that $\chi \otimes \gamma$ is a character of $(L^H)^{\times}$, the fixed field of $H$; to apply this theorem we examine $(\chi \otimes \gamma)|_{(L^{H'})^{\times}}$, where $L^{H'}$ is fixed field of $H'=\langle \iota, \Frob^{f/2} \rangle$ and is the unique quadratic subfield of $L^H$ containing $K$. 
	
Let $\mu=\Ind_H^{H'} \chi \otimes \gamma$. By the determinant formula \cite[(1)]{Gal65}, we observe that $\det \mu= \sign_{[H':H]} \otimes (\chi \otimes \gamma)|_{(L^{H'})^{\times}}$ so we only need to compute $\det \mu$. Choosing representatives $\{1,\Frob^{f/2}\}$ for $H'/H$, we find that $$\mu(\iota) = \begin{pmatrix}
	\chi(\iota) & 0 \\
	0 & \chi(\iota^{-1})
\end{pmatrix}, \qquad \mu(\Frob^{f/2})=\begin{pmatrix}
	0 & \gamma(\Frob^f) \\
	1 & 0
\end{pmatrix},$$
where we have used that $q^{f/2} \equiv -1 \bmod{e}$ since $\theta$ is self-dual (cf. Lemma \ref{induce}iv). A direct computation of $\det \mu \otimes \sign_{[H':H]}$ now shows that $(\chi \otimes \gamma) |_{(L^{H'})^{\times}}= \begin{cases}
	\mathds{1}_{H'} &\text{ if } \theta \text{ is orthgonal,} \\
	\sign_{[H':H]} &\text{ if } \theta \text{ is symplectic.}
	\end{cases}$

\begin{lemma}
\label{FQ}
	Let $\theta$ be an irreducible, self-dual, tamely ramified Weil representation with $\theta=\Ind \chi \otimes \gamma$ as in Lemma \ref{induce}. Let $u \in L^H$ be such that $L^H=L^{H'}(u)$ with $u^2 \in L^{H'}$.
	\begin{enumerate}
		\item If $\theta$ is orthogonal, then $W(\theta)=\chi(u)$.
		\item If $\theta$ is symplectic, then $W(\theta)=-\chi(u)$.
	\end{enumerate}
\end{lemma}

\begin{proof}
Recall from above that $W(\theta)=W(\chi \otimes \gamma)$, with $\chi \otimes \gamma$ a character of $(L^H)^{\times}$. We shall apply the theorem of Fr\"{o}hlich--Queyrut with the quadratic extension $L^H/L^{H'}$. Since $L^H/L^{H'}$ is unramified, we may assume that $u \in \mathcal{O}_{L^{H'}}^{\times}$.

i): Suppose $\theta$ is orthogonal. Then $(\chi \otimes \gamma) |_{(L^{H'})^{\times}}$ is trivial and hence $W(\chi \otimes \gamma)=(\chi \otimes \gamma)(u) = \chi(u)$ since $\gamma=\mathds{1}$ by Lemma \ref{induce}v.

ii): Now let $\theta$ be symplectic. Then $(\chi \otimes \gamma \otimes \sign_{[H':H]}) |_{(L^{H'})^{\times}}$ is trivial and hence $$W(\chi \otimes \gamma \otimes \sign_{[H':H]})= (\chi \otimes \gamma \otimes \sign_{[H':H]})(u)=\chi(u),$$ where for the second equality we use that $(\gamma \otimes \sign_{[H':H]})(u)=1$ since $\gamma \otimes \sign_{[H':H]}$ is unramified and $u \in \mathcal{O}_{L^{H'}}$. On the other hand, by Theorem \ref{props}iii, we have that $$W(\chi \otimes \gamma \otimes \sign_{[H':H]})=W(\chi \otimes \gamma)\sign_{[H':H]}(\Frob^{f/2})=-W(\chi \otimes \gamma),$$ from which we see that $W(\theta)=-\chi(u)$.
\end{proof}

Finally we derive criteria for determining $\chi(u)$, where $L^H=L^{H'}(u)$ with $u^2 \in L^{H'}$. Recall that we may suppose $u \in \mathcal{O}_{L^H}^{\times}$ since $L^H/L^{H'}$ is unramified. Hence we only require $\chi$ to be trivial on $\mathcal{O}_{L^{H'}}^{\times}$; this is true since an unramified twist of $\chi$ is trivial on $(L^{H'})^{\times}$.

\begin{lemma}
\label{FQu}
	Let $k_H, k_{H'}$ be the residue fields of the fixed fields $L^H$ and $L^{H'}$ respectively and let $\tilde{q}=|k_{H'}|$. Let $\chi$ be a tamely ramified character of $\left(L^H\right)^{\times}$ which is trivial on $\mathcal{O}_{L^{H'}}^{\times}$. Suppose $\chi$ has order $e$ on $\mathcal{O}_{L^H}^{\times}$. If $\tilde{q}$ is even, then $\chi(u)=1$. Otherwise
	\begin{equation*}
		\chi(u)=1 \Leftrightarrow v_2(\tilde{q}+1) \geqslant v_2(e) + 1,
	\end{equation*}
	where $v_2$ is the $2$-adic valuation.
\end{lemma}
	
\begin{proof}
	If $\tilde{q}$ is even, then $e$ is odd and hence $\chi(u)=\chi(u^e)=\chi^e(u)=1$, so we now suppose $\tilde{q}$ is odd; this idea further allows us to reduce to the case that $e$ is a power of $2$. As $\chi$ is tamely ramified, it is trivial on principal units and hence we shall view its action on the quotient instead and identify $u,u^2$ with their images after an isomorphism to $k_H^{\times}$ and $k_{H'}^{\times}$.

	Now note $k_{H'} \cong \mathbb{F}_{\tilde{q}}$, $k_{H} \cong \mathbb{F}_{\tilde{q}^2}$ as $L^H/L^{H'}$ is unramified. Let $k_H^{\times} = \langle \xi \rangle$ and observe that $k_{H'}^{\times}=\langle \xi^{\tilde{q}+1} \rangle$. Let $u=\xi^a$. As $u^2 \in k_{H'}$, $(\tilde{q}+1)|2a$ and moreover $(\tilde{q}+1) \nmid a$ since $u \not\in k_{H'}$. Therefore $v_2(\tilde{q}+1)=v_2(2a)=v_2(a)+1$.
		
	As $\chi$ has order $e$, we have $$\chi(u)=1 \Leftrightarrow u \in \left( k_{H}^{\times} \right)^e \Leftrightarrow v_2(e) \leqslant v_2(\tilde{q}+1)-1.$$
\end{proof}

\begin{remark}
	We used $\tilde{q}$ for the size of the residue field to remind the reader that this need not be equal to $q$ since we have had to move to an intermediate field. In fact, if $\theta$ is $f$-dimensional, then $\tilde{q}=q^{f/2}$.
\end{remark}

We now combine the previous two lemmas to give a complete description of $W(\theta)$.

\begin{theorem}
\label{selfdual}
	Let $\theta$ be an irreducible, self-dual, tamely ramified Weil representation of dimension $f \geqslant 2$ such that $|\theta(I)|=e$. 
	
	If $q$ is even, then $W(\theta)=\begin{cases}
		1 &\text{if $\theta$ is orthogonal;} \\
		-1 &\text{if $\theta$ is symplectic.}
	\end{cases}$

	If $q$ is odd, then $W(\theta)=1 \Leftrightarrow \begin{cases}
		v_2(q^{f/2}+1) \neq v_2(e) & \text{if $\theta$ is orthogonal;} \\
		v_2(q^{f/2}+1) = v_2(e) & \text{if $\theta$ is symplectic.}
	\end{cases}$
\end{theorem}

\begin{proof}
	If $q$ is even, then the result follows immediately from Lemmas \ref{FQ} and \ref{FQu}, so assume $q$ is odd. We shall prove the symplectic case; the proof of the orthogonal case is analogous. By the above lemma, we have $\chi(u)=-1 \Leftrightarrow v_2(q^{f/2}+1) \leqslant v_2(e)$. Recall that since $\theta$ is self-dual, we necessarily have $q^{f/2} \equiv -1 \mod{e}$ and hence $v_2(e) \leqslant v_2(q^{f/2}+1)$ which proves the theorem.
\end{proof}

So far we have focused on the irreducible summands but we now close this section by collating such summands to connect the root numbers to certain Legendre symbols depending on $e$. We merely provide the set up and result here and give more detail in Appendix \ref{Jacobiappendix}.

\begin{definition}[=Lemma--Definition \ref{rhoe}]
	Let $\rho$ be a symplectic, tamely ramified Weil representation with finite image of inertia such that the characteristic polynomial of $\rho(\iota)$ has coefficients in $\Z$. Suppose $\rho(\iota)$ has an eigenvalue of order $e$. Then there exist irreducible summands $\rho_1, \cdots, \rho_m$ of $\rho$, $m \geqslant 1$, satisfying the following:
	\begin{enumerate}
		\item $\dim \rho_1= \cdots = \dim \rho_m$,
		\item $m\dim \rho_1=\tilde{\varphi}(e)$,
		\item The $\tilde{\varphi}(e)$ eigenvalues of $\{ \rho_j(\iota) : 1 \leqslant j \leqslant m \}$ all have order $e$ and moreover:
		\begin{enumerate}
			\item If $e \leqslant 2$, then there is exactly one eigenvalue (with multiplicity $2$).
			\item If $e \geqslant 3$, then all $\tilde{\varphi}(e)$ eigenvalues are distinct.
		\end{enumerate}
	\end{enumerate}
	
	We define $\rho_e=\bigoplus\limits_{j=1}^m \rho_j$.
\end{definition}
	
\begin{theorem}[=Theorem \ref{goodrootpf}]
\label{goodroot}
Let $\rho_e$ be a symplectic Weil representation satisfying the conditions above and let $q$ be the cardinality of the residue field of $K$. Then

$W(\rho_e) = \begin{cases}
	\left( \dfrac{q}{l} \right) &\text{ if } e=l^k \quad \text{ for some odd prime $l$ and integer $k\geqslant 1$}; \\[10pt]	
	\left( \dfrac{-1}{q} \right) &\text{ if } e=2l^k \quad \text{ for some prime } l \equiv 3 \bmod{4}, k\geqslant 1,  \qquad \text{ or } e=2; \\[10pt]	
	\left( \dfrac{-2}{q} \right) &\text{ if } e=4; \\[10pt]	
	\left( \dfrac{2}{q} \right) &\text{ if } e=2^k \quad \text{ for some integer } k \geqslant 3; \\[10pt]	
	\quad 1 &\text{ else.}
\end{cases}$
\end{theorem}
	
\begin{remark}
\label{qsquare}
	Note that if the residue degree $f(K/\qp)$ is even, then $q$ is a square hence $W(\rho_e)=1$.
\end{remark}
	
\section{Potentially totally toric reduction}
	We now compute the root number $W(\rho_T \otimes \xcyc \otimes \sp (2))=W(\rho_T \otimes \sp(2))$. Recall that $\rho_T$ is defined over $\Z$ by Fact \ref{decomp}, so in particular has real character.
	
\begin{lemma}\cite[p.14]{Roh96}
\label{mult}
	Let $\tau$ be a representation of $\Gal(K^{\rm{sep}}/K)$ with real character. Then $$W(\tau \otimes \sp(2))= (-1)^{\langle \mathds{1},\tau \rangle}((\det \tau)(-1)).$$
\end{lemma}
	
When $\rho_T$ is tamely ramified, one can decompose $\rho_T$ and apply Theorem \ref{dual} and Corollary \ref{dualrep}ii to compute $(\det \rho_T)(-1)$, which only requires knowledge of the action of inertia on $\rho_T$. We do however need to know the multiplicity of $\mathds{1}$ hence we also need some information about the action of $\rho_T$ on Frobenius. To this end, we use the Euler factor and compute the multiplicity of $x-1$ of the local polynomial.
	
So far we have not used the assumption that $\rho_T$ is tamely ramified in this section; indeed Rohrlich gives an explicit formula at $p=3$. We shall do similarly under some constraints to make the result of Lemma \ref{mult} more explicit.
	
\begin{lemma}
\label{multleg}
	Assume that $\rho_T(I)$ is abelian and $p>2$. Let $\chi_1, \chi_2$ be the ramified quadratic characters of $K$. Then $(\det \rho_T)(-1) = \left( \dfrac{-1}{q} \right)^{\langle \rho_T, \chi_1 \oplus \chi_2 \rangle}$.
\end{lemma}

\begin{proof}
	First note that the Artin map gives an isomorphism of $\mathcal{O}_K^{\times}$ with the absolute inertia group so we only need to consider $\rho_T(I)$. Write the restriction of $\rho_T$ to inertia as $\bigoplus\limits_{j=1}^{m} \psi_j$ where $\psi_j$ are characters. Then $(\det \rho_T)(-1) = \left(\bigotimes\limits_{j=1}^m \psi_j \right)(-1)$. Since $\rho_T$ may be defined over $\Z$, $\det \rho_T$ is self-dual. Hence if $\psi_j$ has order at least $3$, there exists $\psi_k$, $k\neq j$, such that $\psi_k=\psi_j^{-1}$. This reduces us to considering only the characters of $\psi_j$ of order at most $2$. Moreover, if $\psi_j = \mathds{1}$, then $\psi_j(-1)=1$ hence we only need to consider quadratic\footnote{By quadratic, we mean order exactly 2.} characters $\chi$ of $I$.
	
	As $p$ is odd, $\chi$ is unique and we have that $\chi(-1)=\left( \dfrac{-1}{q} \right)$ by Lemma \ref{2reduce} and hence $(\det \rho_T)(-1) = \left( \dfrac{-1}{q} \right)^{\langle \rho_T(I), \chi \rangle}$. To finish, we show that $\langle \rho_T(I), \chi \rangle=\langle \rho_T, \chi_1 \oplus \chi_2 \rangle$; by Frobenius reciprocity, it suffices to show that the induction of $\chi$ of the inertia group to the Weil group is also a quadratic ramified character. This follows immediately since $\Ind \chi$ necessarily factors through a group of the form $\langle \Frob \rangle \rtimes (C_2^n)$ which is abelian.
\end{proof}

\begin{theorem}
	Suppose $\rho_T$ is tamely ramified and let $m_T$ be the multiplicity of $-1$ as an eigenvalue of $\rho_T(\iota)$. Then $$W(\rho_T \otimes \sp (2))= (-1)^{\langle \mathds{1}, \rho_T \rangle} \left( \dfrac{-1}{q} \right)^{m_T}.$$
\end{theorem}

\begin{proof}
	Observe that if $\rho_T(\iota)$ has an eigenvalue of $-1$, then this corresponds to one of the two ramified quadratic characters $\chi_1$ and $\chi_2$ being a summand of $\rho_T$ and hence $m_T=\langle \rho_T, \chi_1 \oplus \chi_2 \rangle$. The result now follows from Lemmas \ref{mult} and \ref{multleg}.
\end{proof}

Combining this with Theorem \ref{goodroot}, we have now proved our first main result.
	
\begin{theorem}
\label{rootno}
	Let $A/K$ be an abelian variety over a non-Archimedean local field which has tame reduction. Then $$W(A/K)=\left(\prod\limits_{e \in \N} W_{q,e}^{m_e} \right) (-1)^{\langle \mathds{1}, \rho_T \rangle} W_{q,2}^{m_T},$$ where:

	$W_{q,e} = \begin{cases}
		\left( \dfrac{q}{l} \right) &\text{ if } e=l^k \quad \text{ for some odd prime $l$ and integer $k\geqslant 1$}; \\[10pt]	
		\left( \dfrac{-1}{q} \right) &\text{ if } e=2l^k \quad \text{ for some prime } l \equiv 3 \bmod{4}, k\geqslant 1,  \qquad \text{ or } e=2; \\[10pt]	
		\left( \dfrac{-2}{q} \right) &\text{ if } e=4; \\[10pt]	
		\left( \dfrac{2}{q} \right) &\text{ if } e=2^k \quad \text{ for some integer } k \geqslant 3; \\[10pt]	
		\quad 1 &\text{ else.}
	\end{cases}$
\end{theorem}

\section{An example}
\label{example}
	
We now use the results of Dokchitser, Dokchitser, Maistret and Morgan \cite{DDMM} to obtain the information we need from the Jacobian of a genus two hyperelliptic curve (via a Weierstrass model) to compute the root number under our assumptions. We first give the relevant definitions (see for example \cite[\S 2.1]{AD17}).
		
\begin{definition}
	Let $K/\qp$ be a finite extension and let $f \in \mathcal{O}_K[x]$ be a squarefree monic polynomial with set of roots $R$. Then a cluster $\mathfrak{s} \subset R$ is a nonempty set of roots of the form $R \cap \mathcal{D}$ where $\mathcal{D} \subset K^{\rm{sep}}$ is a $p$-adic disc.
\end{definition}		
		
To compute the root number, we also need some more notation to construct the relevant representations. We choose the normalised valuation on $v$ of $K$  and extend it to $K^{\rm{sep}}$.
		
\begin{definition}
	For a cluster $\mathfrak{s}$ of cardinality at least $2$, we define:
	
\noindent \begin{tabular}{l p{0.8\linewidth}}
$d_{\ss}$ & $=\min \{ v(r-r') \, | \, r,r' \in \ss \}$, called the depth of $\ss$; \\
$\mathfrak{s}^0$ & the set of maximal subclusters of $\ss$ of odd cardinality; \\
$I_{\ss}$ & $=\Stab_I(\ss)$; \\
$\mu_{\ss}$ & $= \sum\limits_{r \in R \setminus \mathfrak{s}} v(r-r_0)$ for any $r_0 \in \ss$; \\
$\epsilon_{\ss}$ & $=\begin{cases}
			0 \text{ -representation of } I_{\ss} \text{ if } |\ss| \text{ is odd}; \\
			\mathds{1} \text{ -representation of } I_{\ss} \text{ if } |\ss| \text{ is even and } \ord_2 ([I:I_{\ss}]\mu_{\ss}) \geqslant 1; \\
			\text{order two character of } I_{\mathfrak{s}} \text{ if } |\ss| \text{ is even and } \ord_2 ([I:I_{\ss}]\mu_{\ss})< 1;
			\end{cases}$ \\
$\lambda_{\ss}$ & $=\frac{1}{2}(\mu_{\ss}+d_{\ss}|\ss^0|)$; \\
$\gamma_{\mathfrak{s}}$ & any character of $I_{\ss}$ of order equal to the prime-to-$p$ part of the denominator of $[I:I_{\ss}]\lambda_{\ss}$ (with $\gamma_{\ss}=\mathds{1}$ if $\lambda_{\ss}=0$); \\
$V_{\mathfrak{s}}$ & $=\gamma_{\mathfrak{s}} \otimes (\C[\mathfrak{s}^0] \ominus \mathds{1}) \ominus \epsilon_{\mathfrak{s}}$.
\end{tabular}
\end{definition}

\begin{theorem}[{\cite[Theorem 1.19]{DDMM}}]
	Let $l \neq p$ be prime. Let $C:y^2=f(x)$ be a hyperelliptic curve over $K$. Then $$H^1_{\acute{e}t}(C/K^{\rm{sep}}, \mathbb{Q}_l) \otimes_{\mathbb{Z}_l} \C \cong H^1_{ab} \oplus (H^1_t \otimes \sp (2))$$ as complex $I$-representations, with
	\begin{equation*}
		 H^1_{ab} = \bigoplus_{X/I} \Ind_{I_{\mathfrak{s}}}^{I} V_{\mathfrak{s}}, \, \, \, \, H^1_t = \bigoplus_{X/I} (\Ind_{I_{\mathfrak{s}}}^{I} \epsilon_{\mathfrak{s}}) \ominus \epsilon_R,
	\end{equation*}
	where $X$ is the set of clusters that are not singletons and that cannot be written as a disjoint union of more than 2 clusters of even size.
\end{theorem}					
	
\begin{remark}
	$H^1_{ab}$ corresponds to the potentially good case whereas $H^1_t$ corresponds to the potentially toroidal case. It is straightforward to see from the definition above that $\Jac(C)$ has potentially good reduction if and only if all clusters except $R$ have odd cardinality.
\end{remark}

\begin{example}
	Let $f=x^6-8x^4-8x^3+8x^2+12x-8$ and let $C/\Q$ be the hyperelliptic curve $y^2=f(x)$. We shall compute the global root number of $\Jac(C)$. Note that it has conductor $13^4$ so $W(\Jac(C)/\Q)=W(\Jac(C)/\Q_{13})$ since the root number at the Archimedean place is positive.
		
	The cluster picture\footnote{The annotated numbers refer to the depths.} we associate to $f$ at $13$ is

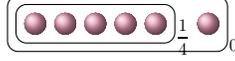
\begin{center}	
\clusterpicture
\Root(0.00,2)B(a1);     
\Root(0.40,2)B(a2);     
\Root(0.80,2)B(a3);      
\Root(1.2,2)B(a4);
\Root(1.6,2)B(a5);
\Root(2.3,2)B(a6);
\ClusterL(c1){(a1)(a2)(a3)(a4)(a5)}{$\dfrac{1}{4}$};   
\ClusterL(c3){(c1)(a6)}{$0$};  
\endclusterpicture			
\end{center}
where we denote the five roots in the inner cluster $\mathfrak{s}_1$ by $\alpha_1,...,\alpha_4,\beta$ and the lone root in the outer cluster $\mathfrak{s}_2$ by $\gamma$. Furthermore, the (cyclic) action of inertia on the roots is given by $( \alpha_1, \alpha_2, \alpha_3, \alpha_4)$.

\begin{center}
\begin{tabular}{l|l|l}
                              & $\mathfrak{s}_1$                  & $\mathfrak{s}_2$ \\ \hline
$d_{\mathfrak{s}}$            & $1/4$                             & $0$ \\
$\mathfrak{s}^0$ & $\{\{\alpha_1\}, \{\alpha_2\}, \{\alpha_3\}, \{\alpha_4\}, \{\beta\} \}$ & $\{\mathfrak{s}_1, \{\gamma\} \}$ \\
$I_{\mathfrak{s}}$            & $I$                               & $I$ \\
$\mu_{\mathfrak{s}}$          & $0$                           & $0$ \\
$\epsilon_{\mathfrak{s}}$     & $0$                               & $\mathds{1}$ \\
$\lambda_{\mathfrak{s}}$      & $5/8$                             & $0$ \\
$\gamma_{\mathfrak{s}}$       & Order $8$ character $\chi$        & $\mathds{1}$ \\
$V_\mathfrak{s}$	  & $\chi \oplus \chi^3 \oplus \chi^5 \oplus \chi^7$ & $0$ representation
\end{tabular}
\end{center}

	Collating this information, we see that for the abelian variety $\Jac(C)/\Q_{13}$, we have $m_8=1$ and $m_e=m_T=0$ otherwise. Therefore $$W(\Jac(C)/\Q)=W_{13,8}=\left(\frac{2}{13} \right) =-1,$$ which implies that the Jacobian has odd (and therefore positive) rank, assuming the parity conjecture. In fact, it is known \cite[label 28561.a.371293.1]{BSSVY16,LMFDB} to have analytic rank $1$ which agrees with our calculation.
\end{example}

\begin{remark}
\label{BKS}
	In a recent paper of Brumer, Kramer and Sabitova \cite{BKS18}, they compute root numbers of several genus two hyperelliptic curves. Our results combined with the above method of \cite{DDMM} enable us to easily compute all the relevant local root numbers away from $3$. At $p=3$, the machinery of \cite{DDMM} provides us with the representation so we use Lemma \ref{mult} to obtain the local root number and completely recover the computations of \cite{BKS18}.
\end{remark}

\section{The twisted root number}
\subsection{Local twisted root numbers}

We shall now generalise our results by considering the effect on the root number after twisting $\rho_A$ by an Artin representation $\tau_v$ with real character; this implies that the twisted root number will be real.
	
Firstly, we note that Sabitova \cite{Sab07, Sab13} has previously given formulae for twisted root numbers but these are not currently practical for computational purposes so shall adapt them using our theory. For completeness, we shall give the relevant propositions we use here.
	
\begin{proposition}\cite[Proposition 1.5]{Sab07}
\label{Sab0}
	Let $\Z=\langle \Frob \rangle$, $I=\langle \iota | \iota^e \rangle$, $G=I \rtimes \Z$ where $\Frob \iota \Frob^{-1}=\iota^q$ with $q$ coprime to $e$. Let $f$ be the least positive integer such that $q^f \equiv 1 \mod{e}$. Then every irreducible symplectic representation $\lambda$ of $G$ which acts faithfully on $I$, factors through $H=G/\langle \Frob^{2f} \rangle$ and as a representation of $H$ has the form $$\lambda = \Ind_{I \rtimes \Gamma}^{H} \phi,$$ where $\Gamma \cong C_2$ is the subgroup of $\Z/\langle \Frob^{2f} \rangle$ generated by $\Frob^f$ and $\phi$ is a character of $I \rtimes \Gamma$ such that
	\begin{enumerate}
		\item $\phi(\iota)$ is of order $e$,
		\item $f$ is even and $q^{f/2} \equiv -1 \mod{e}$,
		\item $\phi(\Frob^f)=-1$.
	\end{enumerate}	
\end{proposition}
	
\begin{definition}
	Given a representation $\lambda=\Ind_{I \rtimes \Gamma}^{H} \phi$ as above, let $\theta$ denote the quadratic character of $I \rtimes \Gamma$ such that $\theta(\Frob^f)=-1$, $\theta(\iota)=1$. We then define $$\hat{\lambda}=\Inf_H^G \Ind_{I \rtimes \Gamma}^{H} (\phi \otimes \theta),$$ where $\Inf_H^G$ is the inflation map from $H$ to $G$. If $\rho$ is an irreducible Weil representation which is not symplectic, we define $\hat{\rho}$ to be the zero representation and then extend this definition linearly to all Weil representations.
\end{definition}

The above proposition is already encapsulated in Lemma \ref{induce} and the symplectic representations are precisely those of the form $\theta=\Ind \chi \otimes \gamma$ with $\gamma \neq \mathds{1}$. This allows us to see directly that $\hat{\theta}=\Ind \chi$; i.e. the corresponding orthogonal representation. The reason we need to consider this twist is due to the following proposition.

\begin{proposition}\cite[Proof of Proposition 1.9]{Sab07}
\label{newSab}
	Let $\theta$ be a tamely ramified, irreducible, symplectic Weil representation and let $\tau_v$ be a self-dual Artin representation of $\Gal(K^{\rm{sep}}/K)$. Then $$W(\theta \otimes \tau_v)=((\det \tau_v)(-1))^{\frac{1}{2} \dim \theta} \chi(u)^{\dim \tau_v} (-1)^{\langle \mathds{1}, \tau_v \rangle + \langle \eta_v, \tau_v \rangle + \langle \hat{\theta}, \tau_v \rangle},$$ where $\eta_v$ is the unramified quadratic character of $K^{\times}$ and $\chi(u)$ is precisely the factor occurring in the Theorem of Fr\"{o}hlich and Queyrut.
\end{proposition}

Our first step towards a complete formula for the twisted root number is to give an explicit description of $\hat{\theta}$ above. Continuing with our approach of breaking $\rho_B$ into symplectic summands of the form $\rho_e$, we shall describe the corresponding orthogonal representations attached to each $\rho_e$. We concentrate for the moment on the case that each irreducible summand is self-dual of (even) dimension at least $2$ (cf. Lemma \ref{induce}) which implies that $\zeta_e \not\in K$.

Observe that if $\theta$ is symplectic of dimension $f$, then $\hat{\theta}$ factors through a Galois extension $L/K$ with $\Gal(L/K)=\langle \iota, \Frob \, | \, \iota^e, \Frob^f, \Frob \iota \Frob^{-1} =\iota^q \rangle \cong C_e \rtimes C_f$ with a faithful action. We now explicitly describe $L$; the proof follows from Proposition \ref{Sab0}.

\begin{lemma}
	Let $K$ be a non-Archimedean local field with residue cardinality $q$. Let $\theta$ be an irreducible, orthogonal, tamely ramified Weil representation of $\mathcal{W}(K^{\rm{sep}}/K)$ such that $\dim \theta \geqslant 2$ and $|\theta(I)|=e$. Then $\theta$ factors faithfully through $\Gal(K(\zeta_e,\pi_K^{1/e})/K)$.
\end{lemma}

With our extension $L/K$ determined, we may now describe the $\hat{\theta}$ as Galois representations and not just abstract representations.

\begin{lemma}
	Let $\rho$ be a symplectic, tamely ramified Weil representation of the form $\rho_e$ (cf. Definition \ref{rhoe}) for some integer $e>2$. Assume moreover that every irreducible summand of $\rho$ is symplectic. Then the direct sum of the orthogonal twists, $\hat{\rho}_e$, is isomorphic to the direct sum of all irreducible, faithful, $f$-dimensional representations of $\Gal(K(\zeta_e,\pi_K^{1/e})/K)$, where $f=[K(\zeta_e):K]$ is the order of $q \bmod{e}$.
\end{lemma}

\begin{proof}
	First observe that for each irreducible summand $\theta_j$ of $\rho$, $\hat{\theta}_j$ is faithful and irreducible since it is a one-dimensional twist by the push-pull formula. Now note that distinct summands $\theta_j$ of $\rho$ have distinct traces on $\iota$ and hence $\theta_j$ are also all distinct.
	
	This produces $\tilde{\varphi}(e)/f$ such representations; to finish the proof, we simply need to show that $\Gal(L/K)$ has the same number of representations with these properties, where $L=K(\zeta_e,\pi_K^{1/e})$; this is a straightforward exercise.
\end{proof}

The above lemma prompts us to make the following definition.

\begin{definition}
\label{rhoef}
	Let $L=K(\zeta_e,\pi_K^{1/e})$ and $f=f(L/K)=[K(\zeta_e):K]$. Then $\rho_{e,f}$ is the direct sum of all irreducible, faithful, $f$-dimensional representations of $\Gal(L/K)$.
\end{definition}

However, since $\hat{\theta}$ only sees irreducible symplectic summands, we a priori need to distinguish these. This turns out to be unnecessary if we only care about multiplicities $\!\!\! \mod{2}$, which will be true for twisted root numbers (see Lemma \ref{twistlem}).

\begin{lemma}
\label{rhoeftau}
	Let $\rho=\bigoplus\limits_j \rho_{e_j}$ be tamely ramified, symplectic Weil representation where $\rho_{e_j}$ are of the form $\rho_e$. Then for all self-dual Artin representations $\tau_v$ of $\Gal(K^{\rm{sep}}/K)$: $$\langle \hat{\rho}, \tau_v \rangle \equiv \sum\limits_j \langle \rho_{e_j,f_j}, \tau_v \rangle \bmod{2},$$ where $f_j$ is the order of $q \bmod{e_j}$.
\end{lemma}

\begin{proof}
	First note that if every irreducible summand of $\rho$ is symplectic, then we have equality. All irreducible summands $\theta$ of $\rho$ which are not symplectic give no contribution to $\hat{\rho}$, so we must show that they always occur with even multiplicity in $\tau$. Observe that we may restrict ourselves to $\theta$ of Galois type which are not symplectic.
	
	By Lemma \ref{semisymp}, $\theta \oplus \theta^*$ is a subrepresentation of $\rho$. Moreover, for any self-dual Artin representation $\tau_v$, $\langle \theta, \tau_v \rangle = \langle \theta^*, \tau_v \rangle.$ Hence $
	\sum\limits_j \langle \rho_{e_j,f_j}, \tau_v \rangle = \langle \hat{\rho}, \tau_v \rangle + 2\langle \theta_1, \tau_v \rangle \equiv \langle \hat{\rho}, \tau_v \rangle \bmod{2},$ for some $\theta_1$.
\end{proof}

With the $\hat{\rho}$ described, we can now focus on the twisted root numbers themselves.

\begin{lemma}
\label{twistlem}
	Let $\rho$ be a tamely ramified Weil representation and let $\rho_T, \tau_v$ be self-dual Artin representations. Then:
	\begin{enumerate}
		\item $W(\rho_T \otimes \tau_v \otimes \sp(2))=W(\rho_T \otimes \sp(2))^{\dim \tau_v} ((\det \tau_v)(-1))^{\dim \rho_T} (-1)^{\langle \rho_T, \tau_v \rangle + \dim \tau_v \langle \mathds{1}, \rho_T \rangle}$;
		\item $W((\rho \oplus \rho^*) \otimes \tau_v)=W(\rho \oplus \rho^*)^{\dim \tau_v} ((\det \tau_v)(-1))^{\dim \rho}$;
		\item If $\rho$ is irreducible and symplectic, then $$W(\rho \otimes \tau_v)=W(\rho)^{\dim \tau_v} ((\det \tau_v)(-1))^{\frac{1}{2} \dim \rho} (-1)^{\dim \tau_v + \langle \mathds{1}, \tau_v \rangle + \langle \eta_v, \tau_v \rangle + \langle \hat{\rho}, \tau_v \rangle},$$ where $\eta_v$ is the quadratic unramified character of $K^{\times}$.
	\end{enumerate}
\end{lemma}

\begin{proof}
(i): This is a simple application of Lemma \ref{mult}. 

(ii): Note $(\rho \oplus \rho^*) \otimes \tau_v \cong (\rho \otimes \tau_v) \oplus (\rho \otimes \tau_v)^*$ so we use Corollary \ref{dualrep}.

(iii): By Proposition \ref{newSab}, $W(\rho \otimes \tau)=((\det \tau)(-1)^{\frac{1}{2} \dim \rho} \chi(u)^{\dim \tau} (-1)^{\langle \mathds{1}, \tau_v \rangle + \langle \eta_v, \tau_v \rangle + \langle \hat{\rho}, \tau_v \rangle}.$ Making the substitution $W(\rho)=-\chi(u)$ (cf. Lemma \ref{FQ}) yields the result.
\end{proof}

\begin{theorem}
\label{twistthm}
	Let $A/K$ be an abelian variety over a non-Archimedean local field which has tame reduction and let $\tau_v$ be a self-dual Artin representation of $\Gal(K^{\rm{sep}}/K)$. Write $\rho_A= \rho_B \oplus (\rho_T \otimes \xcyc^{-1} \otimes \sp(2))$ as in Fact \ref{decomp}. Moreover, write $\rho_B \otimes \xcyc^{1/2} = \bigoplus\limits_j \rho_{e_j}$ as a decomposition of summands of the form $\rho_e$ (cf. Definition \ref{rhoe}) and let $m_e=\{ j: e_j=e \}$ count the multiplicity of such summands. 
	
	 Then $$W(A/K, \tau_v)=W(A/K)^{\dim \tau_v}((\det \tau_v)(-1))^{\dim A}(-1)^{l_1+l_2},$$ where $$l_1=\langle \rho_T,\tau_v \rangle + \dim \tau_v \langle \mathds{1}, \rho_T \rangle,$$ $$l_2= \sum\limits_{e \in \N} m_e \left( \langle \rho_{e,f},\tau_v \rangle + \dfrac{\tilde{\varphi}(e)}{[K(\zeta_e):K]}(\langle \mathds{1},\tau_v \rangle + \langle \eta, \tau_v \rangle + \dim \tau_v)\right),$$ with $\eta_v$ the unramified quadratic character of $K^{\times}$ and for a fixed $e \in \N$, $f=[K(\zeta_e):K]$ is the order of $q \bmod{e}$.
\end{theorem}

\begin{proof}
First note that $W(\rho_A \otimes \tau_v)=W(\rho_B \otimes \xcyc^{1/2} \otimes \tau_v)W(\rho_T \otimes \tau_v \otimes \sp(2))$ by Corollary \ref{dualrep} and that $W(\rho_T \otimes \tau_v \otimes \sp(2))$ is completely described by Lemma \ref{twistlem} and recovers the term $l_1$.

We now study $\rho_B \otimes \xcyc^{1/2}=\bigoplus\limits_j \rho_{e_j}$. Observe that for any irreducible summand $\rho$ which is not symplectic, $\rho \oplus \rho^*$ is a subrepresentation and by Lemma \ref{twistlem} $$W((\rho \oplus \rho^*) \otimes \tau_v)=W(\rho \oplus \rho^*)^{\dim \tau_v} ((\det \tau_v)(-1))^{\dim \rho} (-1)^{2\dim \tau_v + 2\langle \mathds{1}, \tau_v \rangle + 2\langle \eta_v, \tau_v \rangle + 2\langle \hat{\rho}, \tau_v \rangle},$$ since $\hat{\rho}=\hat{\rho}^*$ is the zero representation.
	
Since $\dfrac{W((\rho \oplus \rho^*) \otimes \tau_v)}{W(\rho \oplus \rho^*)^{\dim \tau_v}}$ only depends on $\dim \rho$ and the non-symplectic summands always arise in this way, we obtain $$\dfrac{W(\rho_B \otimes \xcyc^{1/2} \otimes \tau_v)}{W(\rho_B \otimes \xcyc^{1/2})^{\dim \tau_v}}= ((\det \tau_v)(-1))^{\frac{1}{2}\dim \rho_B} (-1)^{m \dim \tau_v + m \langle \mathds{1},\tau_v \rangle + m \langle \eta_v, \tau_v \rangle + \langle \bigoplus_j \hat{\rho}_{e_j}, \tau_v \rangle},$$ where $m$ is the number of irreducible summands of $\rho_B \otimes \xcyc^{1/2}$.

For a fixed subrepresentation $\rho_e$, observe that each irreducible summand has dimension  $f=[K(\zeta_e):K]$ and hence $\rho_e$ has $\frac{\tilde{\varphi}(e)}{[K(\zeta_e):K]}$ summands. Moreover, there are $m_e$ such representations of the form $\rho_e$ so we get $$\dfrac{W(\rho_B \otimes \xcyc^{1/2} \otimes \tau_v)}{W(\rho_B \otimes \xcyc^{1/2})^{\dim \tau_v}}= ((\det \tau_v)(-1))^{\frac{1}{2}\dim \rho_B} (-1)^l,$$ where $l=\langle \bigoplus_j \hat{\rho}_{e_j},\tau_v \rangle + \sum\limits_{e \in \N} m_e \left( \dfrac{\tilde{\varphi}(e)}{[K(\zeta_e):K]}(\langle \mathds{1},\tau_v \rangle + \langle \eta, \tau_v \rangle + \dim \tau_v)\right)$. Furthermore, $(-1)^{\langle \bigoplus_j \hat{\rho}_{e_j}, \tau_v \rangle}=(-1)^{\sum_j \langle \rho_{e_j,f_j}, \tau_v \rangle}$ by Lemma \ref{rhoeftau} hence $l=l_2$. Collating the above computations completes the proof.
\end{proof}

\subsection{The global case}
	We now wish to look at the contribution of a global twist $\tau$, using our relation in Theorem \ref{twistthm}. To consider the global contribution, we should also take into account the infinite places (cf. Remark \ref{infty}). Observe that for each finite place, we obtained a factor of $((\det \tau_v)(-1))^{\dim A}$. Viewing $\det \tau$ as an adelic character which is trivial on $\mathbb{Q}^{\times}$ instead, we see that $$\sign(\det \tau):= \prod\limits_{v| \infty, \, v \in M_{\KK}} (\det \tau_v)(-1) = \prod\limits_{v < \infty, \, v \in M_{\KK}} (\det \tau_v)(-1).$$
	
	Recall that we define $$W(A/\KK,\tau):=\prod\limits_{v \in M_{\KK}} W(A/\KK_v,\tau_v),$$ where $M_{\KK}$ is the set of all places of $\KK$.
	
\begin{theorem}
\label{globtwist}
	Let $\KK$ be a global field, $A/\KK$ an abelian variety and $\tau$ a finite dimensional Artin representation with real character. Let $M_{\KK}$ be the set of places of $\KK$. For each finite place $v \in M_{\KK}$, write $\rho_{A/\KK_v}=\rho_{B_v} \oplus (\rho_{T_v} \otimes \xcyc^{-1} \otimes \sp(2))$ where $\rho_{B_v},\rho_{T_v}$ have finite image of inertia. If $\tau_v$ is ramified, assume ${A/\KK_v}$ has tame reduction. Then $$W(A/\KK, \tau) = W(A/\KK)^{\dim \tau} (\sign (\det \tau))^{\dim A} \cdot T \cdot S,$$ where
	\begin{eqnarray*}
	\sign(\det \tau) &=& \prod\limits_{v| \infty, \, v \in M_{\KK}} (\det \tau_v)(-1), \\
	T &=& \prod\limits_{v<\infty, v \in M_{\KK}}(-1)^{\langle \rho_{T_v},\tau_v \rangle + \dim \tau \langle \mathds{1}, \rho_{T_v} \rangle}, \\
	S &=& \prod\limits_{v<\infty, v \in M_{\KK}} \prod\limits_{e \in \N} \left( (-1)^{\langle \rho^v_{e,f}, \tau_v \rangle + \frac{\tilde{\varphi}(e)}{[\KK_v(\zeta_e):\KK_v]}(\langle \mathds{1},\tau_v \rangle + \langle \eta_v, \tau_v \rangle + \dim \tau)}\right)^{m_{e,v}},
	\end{eqnarray*}
	$m_{e,v}$ is the multiplicity of representations $\rho_e$ which are subrepresentations of $\rho_{B_v}$ (as defined previously in Definition \ref{rhoe}), $\rho^v_{e,f}$ is the Artin representation $\rho_{e,f}$ for $\Gal(\KK_v^{\rm{sep}}/\KK_v)$ (Definition \ref{rhoef}) and $\eta_v$ is the unramified quadratic character of $\KK_v^{\times}$.
\end{theorem}

\begin{proof}
	This is a straightforward consequence of Theorem \ref{twistthm}.
\end{proof}

\begin{remark}
	Note that in the dual case, both $\tilde{\varphi}(e)/[\KK_v(\zeta_e):\KK_v]$ and $\langle \rho^v_{e,f},\tau_v \rangle$ are even so $S$ has no contribution from such summands (cf. Lemma \ref{induce}iv).
\end{remark}

\begin{corollary}\cite[Proposition 1]{Sab13}
\label{coprimetwist}
	Let $\KK$ be a global field, $A/\KK$ an abelian variety and $\tau$ a self-dual Artin representation of $\Gal(\KK^{\rm{sep}}/\KK)$. Assume the conductor $\mathfrak{N}$ of $A/\KK$ is coprime to the conductor of $\tau$. Then $$W(A/\KK, \tau) = W(A/\KK)^{\dim \tau} ((\det \tau)(\mathfrak{N})) (\sign (\det \tau))^{\dim A}.$$
\end{corollary}

\subsection{Recovering $\rho_A$}
Before considering the twisted root number globally, we shall use the theory we've developed so far to reconstruct certain summands of $\rho_A$.
	
\begin{proposition}
	Let $A/K$ be an abelian variety and suppose that $\rho_A$ is tamely ramified. Write $\rho_A=\rho_B \oplus (\rho_T \otimes \xcyc^{-1} \otimes \sp(2))$ as in Fact \ref{decomp} and assume the eigenvalues (or their orders) including multiplicity of $\rho_B(\iota), \rho_B(\iota)$ are known. Then there exists a unique isomorphism class for $\rho_A$ as a complex inertia-representation.
\end{proposition}

\begin{proof}
	This is trivial from the fact that $\rho_A(I)$ is necessarily abelian.
\end{proof}

We briefly mention that we can reconstruct a symplectic representation of the form $\rho_e$ from our $\rho_{e,f}$ in the self-dual case and relate this to a subrepresentation of $\rho_B \otimes \xcyc^{1/2}$.

\begin{proposition}
	Let $e \in \N$ be such that $p \nmid e$ and the order $f$ of $q \bmod{e}$ is such that $f \geqslant 2$ is even and $q^{f/2}\equiv -1 \bmod{e}$. Let $\nu$ be an unramified character of order $2f$. Then $\rho_{e,f} \otimes \nu$ is a symplectic Weil representation $\rho$ of the form $\rho_e$ such that every irreducible summand is also symplectic.
	
	Moreover, if $B/K$ is an abelian variety with tame, potentially good reduction such that $\rho_B(\iota)$ has an eigenvalue $e$, then $\rho_B \otimes \xcyc^{1/2}$ has a subrepresentation of the form $\rho_e$ built from the irreducible summands of $\rho$ and $\rho_{e,f}$. In particular, if every summand of the subrepresentation of the form $\rho_e$ is symplectic, then it is isomorphic to $\rho$.
\end{proposition}

\begin{proof}
	Recall that by construction, every irreducible summand of $\rho_{e,f}$ is orthogonal and hence every summand of $\rho_{e,f} \otimes \nu$ is symplectic by Lemma \ref{FSlemma}. Since the conditions on $e$ force irreducible summands to be self-dual and there are exactly two choices (a symplectic or orthogonal choice), the rest of the proposition follows.
\end{proof}

\section{All quadratic twists with equal parity}
\label{quad}
There are three different notions of parity for an abelian variety: analytic parity via the root number; parity of the rank of the Mordell-Weil group; and parity of the $p^{\infty}$-Selmer group for a given prime $p$ (referred to as $p$-parity). These are equivalent subject to the conjectures of Shafarevich-Tate and Birch--Swinnerton-Dyer. The equivalence of analytic parity and $p$-parity has been proven for elliptic curves over $\Q$ \cite[Theorem 1.4]{DD10}; Morgan has also shown equivalence to $2$-parity for Jacobians of hyperelliptic curves over particular quadratic extensions \cite[Theorem 1.1]{Mor15}.

\begin{example}\cite[Example 9.2]{MR10}
	Let $E/\KK$ be an elliptic curve over a number field with complex multiplication defined over $\KK$. Then the global root number of any quadratic twist $W(E'/\KK)$ is equal to $W(E/\KK)$; the same statement for $2$-parity and parity of the ranks of $E'(\KK)$ is true.
\end{example}

Mazur and Rubin have previously determined necessary conditions for an elliptic curve whose quadratic twists all have equal $2$-Selmer parity: $\KK$ must be totally imaginary and $E/\KK$ has good or additive reduction everywhere \cite[Theorem 9.5]{MR10}. On the other hand, the Dokchitsers have shown that this $2$-parity phenomenon holds for elliptic curves if and only if the equivalent root number statement does \cite[Corollary 1.6]{DD11}. They have further derived necessary and sufficient conditions in this case \cite[Theorem 1]{DD09}. We now extend their result to abelian varieties, continuing in the terminology of \cite{DD09}.

\begin{definition}
	Let $A/\KK$ be an abelian variety over a local or global field. Then we call $A/\KK$ \emph{lawful} if $W(A/\FF)=1$ for every quadratic extension $\FF/\KK$. We say a curve is lawful if its Jacobian is.
\end{definition}

Observe that if $\KK$ is a global field, then $A/\KK$ being lawful is equivalent to all quadratic twists of $A$ having the same root number. Moreover, $A/\KK$ is lawful if and only if $A/\KK_v$ is lawful for all places $v$ of $\KK$.\footnote{If $A/\KK_v$ is not lawful for some $v$, then by imposing only finitely many local conditions we can find a quadratic extension of $\KK$ with negative root number.} Lawful abelian varieties $A/\KK$ come in two flavours depending on $W(A/\KK)$: \emph{lawful evil} if $W(A/\KK)=-1$ and \emph{lawful good} if $W(A/\KK)=1$. Note that the lawful evil case implies that the Mordell-Weil rank must increase in \emph{every} quadratic extension. We list some examples of lawful genus two hyperelliptic curves in Appendix \ref{lawfultable}.

\begin{lemma}
\label{inflaw}
	Let $\KK$ be a number field and let $A/\KK$ be lawful. Then either $\dim A$ is even or $\KK$ has no real places.
\end{lemma}

\begin{theorem}
\label{lawful}
	Let $A/K$ be an abelian variety over a non-Archimedean local field. Write $\rho_{A/K} = \rho_{B/K} \oplus (\rho_{T,K} \otimes \xcyc^{-1} \otimes \sp(2))$. Assume that $\rho_{B/K}$ is tamely ramified and $\rho_T(I)$ is abelian such that if $A/K$ doesn't have potentially good reduction, then the cardinality $q$ of the residue field of $K$ is odd. Write $$\rho_{B/K} \otimes \xcyc^{1/2} = \bigoplus\limits_{e \in \mathbb{N}} \rho_e^{m_e}.$$
	
	If $q$ is odd, we moreover let $\eta_K$ be the unramified quadratic character of $K$ and let $\chi_1, \chi_2$ be the ramified quadratic characters of $K$ and define $$\langle \rho_{T,K}, \mathds{1} \rangle = n_1, \qquad \langle \rho_{T,K}, \eta_K \rangle = n_2, \qquad \langle \rho_{T,K}, \chi_1 \rangle = n_3, \qquad \langle \rho_{T,K}, \chi_2 \rangle = n_4.$$
	
	Let $W_g=\prod\limits_{2 \nmid e} W_{q,e}^{m_e} \prod\limits_{e=4 \, or \, 2||e} W_{q,e/2}^{m_e}$, where $2||e$ means that $v_2(e)=1$, i.e. $e \equiv 2 \bmod{4}$.
	 
	(i) If $p=2$, then $A/K$ is lawful if and only if $W_g=1$.
	
	(ii) If $q$ is odd, then $A/K$ is lawful if only if $n_1 \equiv n_2 \mod{2},$ $n_3 \equiv n_4 \mod{2},$ and $W_g=(-1)^{n_1+n_3}$.
\end{theorem}

\begin{proof}
	Let $F/K$ be a quadratic extension and consider first $\rho_{B/K}$. If $F/K$ is unramified, then $W(\rho_{B/F})=1$ by Remark \ref{qsquare}. We now assume $F/K$ is ramified. Recall that $e$ is the order of the image of inertia of $\rho_e$ and hence the ramification degree of the Galois extension that $\rho_e$ factors through. If $2 \nmid e$, then the image of inertia still has order $e$ and hence $\Res_{F/K} \rho_e$ is still a representation of the form $\rho_e$ and hence still occurs in $\rho_{B/K}$.
	
	Otherwise, $2|e$ and the order of the image of inertia of each irreducible summand is $e/2$ and hence forms a summand of $\rho_{e/2}$. Furthermore, if $4|e$ with $e>4$, then $\tilde{\varphi}(e)=2\tilde{\varphi}(e/2)$ so $\Res_{F/K} \rho_e$ is isomorphic to the direct sum of two representations of the form $\rho_{e/2}$ and we obtain $W(\Res_{F/K} \rho_e) = W(\rho_{e/2})^2=1$. Therefore $W(\rho_{B/F})=W_g$.
	
	We now focus on $W(\rho_{T,F} \otimes \xcyc^{-1} \otimes \sp(2))=W(\rho_{T,F} \otimes \sp(2))$. Recall that the root number only depends on the multiplicities of the trivial and quadratic ramified characters in $\rho_{T,F}$ by Lemma \ref{multleg}.
	
	If $F/K$ is unramified, then $\langle \rho_{T,F}, \mathds{1} \rangle = n_1 +n_2$. Moreover, $\chi_1$ and $\chi_2$ restrict to the same ramified quadratic character $\chi'_1$ of $F$, hence $\langle \rho_{T,F}, \chi'_1 \rangle = n_3+n_4$. Let $\chi'_2$ be the other quadratic character of $F$. We claim $\langle \rho_{T,F}, \chi'_2 \rangle = \langle \rho_{T,K}, \Ind_{F/K} \chi'_2 \rangle$ is even and hence $W(\rho_{T,F} \otimes \sp(2)) = (-1)^{n_1+n_2} \left(\dfrac{-1}{q} \right)^{n_3+n_4}$.
	
	To prove the claim, let $\eta$ be an order $4$ unramified character of $K$. Then $\Res_{F/K} \eta$ is the unramified quadratic character of $F$ and $\chi'_2 = \chi'_1 \otimes \Res_{F/K} \eta$. By the push-pull formula, we have $\Ind_{F/K} \chi'_2 = (\Ind_{F/K} \chi'_1) \otimes \eta = (\chi_1 \otimes \eta) \oplus (\chi_2 \otimes \eta)$. Now $(\chi_2 \otimes \eta)^{-1}=\chi_1 \otimes \eta$ has order $4$ and hence $\langle \rho_{T,F}, \chi'_2 \rangle = \langle \rho_{T,K}, (\chi_1 \otimes \eta) \oplus (\chi_1 \otimes \eta)^{-1} \rangle$ is even since $\rho_{T,K}$ is self-dual; this completes the unramified case.
	
	Now suppose $F/K$ is ramified and recall that $\rho_{T,K}$ contains a ramified quadratic summand if and only if its restriction to inertia also does. There are two such ramified extensions, corresponding to the two ramified characters $\chi_1$ and $\chi_2$. Without loss of generality, we consider the extension $F$ corresponding to $\chi_1$. In this case, $\langle \rho_{T,F}, \mathds{1} \rangle = n_1+n_3$ and the restriction $\Res_{F/K} \chi_2$ is unramified quadratic. We now claim that $\rho_{T,F}$ contains an even number of each quadratic ramified character of $F$ and hence $W(\rho_{T,F} \otimes \sp(2))=(-1)^{n_1+n_3}$. Similarly, we obtain $W(\rho_{T,F} \otimes \sp(2))=(-1)^{n_1+n_4}$ from the other ramified extension.
	
	To see the claim, let $\chi_F$ be a quadratic character of the inertia group of $F$. Then $\chi_F$ is the restriction of an order 4 character $\tilde{\chi}$ of the inertia group of $K$ since $p$ is odd. Observe that $\tilde{\chi} \neq \tilde{\chi}^{-1}$ but $\Res_{F/K} \tilde{\chi} = \Res_{F/K} \tilde{\chi} = \chi_F$, where the restriction is on the corresponding inertia groups. Since $\rho_{T,K}$ is self-dual, $\langle \rho_{T,K}(I), \tilde{\chi} \rangle = \langle \rho_{T,K}(I), \tilde{\chi}^{-1} \rangle$, where we write $\rho_{T,K}(I)$ for the restriction of $\rho_{T,K}$ to the inertia group of $K$. Hence $\langle \Res_{F/K} \rho_{T,K}(I), \chi_F \rangle = \langle \rho_{T,F}(I), \chi_F \rangle$ is even which proves the claim.
	
	We now derive our lawfulness criteria. Recall that $A/K$ is lawful if $W(A/F)=1$ for all quadratic extensions $F/K$. If $p=2$, then by assumption $A/K$ has potentially good reduction so is lawful if and only if $W_g=1$.
	
	Now assume $q$ is odd. If $F/K$ is unramified, then $W(\rho_{B/F})=1$ and $W(\rho_{T,F})=(-1)^{n_1+n_2} \left(\dfrac{-1}{q} \right)^{n_3+n_4}$. If $q \equiv 1 \mod{4}$, then $W(A/F)=1$ if and only if $n_1 \equiv n_2 \mod{2}$, whereas if $q \equiv 3 \mod{4}$, then we must have $n_1+n_2 \equiv n_3+n_4 \mod{2}$.
	
	If $F/K$ is ramified, then we do not know $W_g$ but note that it does not depend on which ramified extension we have and hence if $A/K$ is lawful, then $W(\rho_{T,F} \otimes \sp(2))$ should not depend on the choice of ramified extension either. This implies that $n_3 \equiv n_4 \mod{2}$ and $W_g=(-1)^{n_1+n_3}$.
	
	Collating this information, we find that in both congruence classes for $q$, it is necessary and sufficient to have that $n_1 \equiv n_2 \mod{2}$, $n_3 \equiv n_4 \mod{2}$ and $W_g=(-1)^{n_1+n_3}$ which proves the theorem.
\end{proof}

\begin{remark}
	When $A=E$ is an elliptic curve, then a result of the Dokchitsers \cite[Theorem 1]{DD09} states that $E/K$ is lawful if and only if it attains good reduction over an abelian extension. This is recoverable from our result under our assumptions. Indeed, if $E$ has potentially multiplicative reduction, then $\dim \rho_T =1$ and $\rho_T$ is a character of order at most 2 hence cannot satisfy our congruence condition. Now suppose $E$ has potentially good reduction. Let $e$ be the order of the eigenvalues of $\rho_E(\iota)$ and note that $E$ attains good reduction over an abelian extension if and only if $q \equiv 1 \mod{e}$ \cite[Proposition 2ii]{Roh93}; this is true if and only if $W_g=1$.
\end{remark}

We briefly note one curious necessary condition for an abelian variety to be lawful.

\begin{remark}
	Obesrve that if $\KK$ is a number field and $A/\KK$ is a lawful abelian variety of conductor $\mathfrak{N}$, then $\mathfrak{N}$ is necessarily a square. Indeed, Lemma \ref{inflaw} and Corollary \ref{coprimetwist} implies $\chi(\mathfrak{N})=1$ for every quadratic character $\chi$ of $\KK$ with conductor coprime to $\mathfrak{N}$ which can only happen if $\mathfrak{N}$ is a square as otherwise one could impose finitely many local conditions to construct a counterexample. This observation is reflected in the table in Appendix \ref{lawfultable}.
\end{remark}

\begin{remark}
	Consider the second lawful evil curve\footnote{The minimal equation for this curve is actually $y^2+(x^3+x+1)y=-3x^4+5x^2-5x+1$ but the cluster machinery we use to obtain the root number requires the curve to be in the form $y^2=f(x)$. Its LMFDB label \cite{LMFDB} is 10609.b.10609.1.} in Appendix \ref{lawfultable}, given by the equation $y^2=x^6-10x^4+2x^3+21x^2-18x+5$ whose Jacobian has conductor $103^2$. This is a genuinely new example that does not arise from the Weil restriction of a lawful elliptic curve; to see this we note that its Jacobian is simple using Stoll's criterion \cite[p.1343-1344]{Sto95} at $p=3$.
\end{remark}

\appendix
\section{Jacobi symbols}
\label{Jacobiappendix}
\subsection{Packaging the representation}
	So far we have concentrated on the irreducible summands but we now collate such summands to connect the root numbers of particular types of representations to certain Jacobi symbols depending on $e$.

\begin{lemdef}
\label{rhoe}
	Let $\rho$ be a symplectic, tamely ramified Weil representation such that the characteristic polynomial of $\rho(\iota)$ has coefficients in $\Z$. Suppose $\rho(\iota)$ has an eigenvalue of order $e$. Then there exist irreducible summands $\rho_1, \cdots, \rho_m$ of $\rho$, $m \geqslant 1$, satisfying the following:
	\begin{enumerate}
		\item $\dim \rho_1= \cdots = \dim \rho_m$,
		\item $m\dim \rho_1=\tilde{\varphi}(e)$,
		\item The $\tilde{\varphi}(e)$ eigenvalues of $\{ \rho_j(\iota) : 1 \leqslant j \leqslant m \}$ all have order $e$ and moreover:
		\begin{enumerate}
			\item If $e \leqslant 2$, then there is exactly one eigenvalue (with multiplicity $2$).
			\item If $e \geqslant 3$, then all $\tilde{\varphi}(e)$ eigenvalues are distinct.
		\end{enumerate}
	\end{enumerate}
	
	We define $\rho_e$ to be any representation of the form $\bigoplus\limits_{j=1}^m \rho_j$. Moreover, $\rho$ can be decomposed into summands of the form $\rho_e$.
\end{lemdef}

\begin{proof}
	Observe that for any irreducible summand $\rho'$ of $\rho$, all eigenvalues of $\rho'(\iota)$ neccesarily have the same order (cf. proof of Lemma \ref{reduce}). Since the characteristic polynomial of $\rho(\iota)$ has coefficients in $\Z$, there exist irreducible representations $\rho_1,\cdots, \rho_{m'}$ such that the characteristic polynomial of $\bigoplus\limits_{j=1}^{m'} \rho_j(\iota)$ is the $e$-th cyclotomic polynomial.
	
	
	If $e \leqslant 2$, then $m'=1$ and $\dim \rho_1=1$l; we are done by applying Lemma \ref{char} (hence $m=2$). We may now assume $e \geqslant 3$. In this case we have all $\tilde{\varphi}(e)$ eigenvalues and they are distinct since the cyclotomic polynomial is separable so $m'=m$. To see that the summands have equal dimension, note that this is controlled by the order of $q \bmod{e}$ due to the relation $\Frob \iota \Frob^{-1}=\iota^q$; indeed the sets of eigenvalues of the $\rho_j(\iota)$ correspond to cosets of the subgroup $\langle q \rangle \subset (\Z/e\Z)^{\times}$.
	
	To see that $\rho$ has a decomposition into summands of this form, note that the rationality of the characteristic polynomial forces the correct multiplicities for $e \geqslant 3$ and the symplectic condition does the same for $e=1,2$ by Lemma \ref{char}.
\end{proof}

For the rest of this section, we wish to suppose that our representations $\rho_e$ are themselves symplectic to apply our results to compute their corresponding root numbers. Our first step in this direction is to prove that they are self-dual, after possibly reordering the summands.

\begin{lemma}
	Let $\rho$ be a symplectic, tamely ramified Weil representation such that the characteristic polynomial of $\rho(\iota)$ has coefficients in $\Z$. Then there exists a decomposition of $\rho=\bigoplus \rho_{e_j}$ such that each $\rho_{e_j}$ is self-dual and satisfies Definition \ref{rhoe}.
\end{lemma}

\begin{lemma}
\label{FSlemma}
	Let $\theta$ be an irreducible, self-dual, tamely ramified Artin representation of $\Gal(K^{\rm{sep}}/K)$ with $\dim \theta \geqslant 2$ and let $\nu$ be any unramified character of order $2\dim \theta$. If $\theta$ is orthogonal (resp. symplectic), then $\theta \otimes \nu$ is symplectic (resp. orthogonal) and moreover $W(\theta \otimes \nu)=-W(\theta)$.
\end{lemma}

\begin{proof}
	Let $f=\dim \theta$, $e=|\theta(I)|>1.$ Then by Lemma \ref{induce}, $\theta$ factors through a Galois extension $L/K$ where $\Gal(L/K)= \langle \iota, \Frob | \iota^e, \Frob^{2f}, \Frob \iota \Frob^{-1} = \iota^q \rangle,$ and $q^f \equiv 1 \bmod{e}$. Now let $\nu'$ be any unramified character for $\Gal(L/K)$. Let $H=\langle \iota, \Frob^f \rangle$ and recall that $\theta=\Ind_{L^H/K} \chi \otimes \gamma$ is monomial and hence by the push-pull formula, we have $\theta \otimes \nu' = \Ind_{L^H/K} (\chi \otimes \gamma \otimes \Res_{L^H/K} \nu')$. 
	
	Let $\nu$ be a primitive unramified character of $\Gal(L/K)$ so $\nu$ has order $2f$ and note that $\Res_{L^H/K} \nu^r = \mathds{1}$ if and only if $r$ is even. Moreover, as $\nu$ is primitive, $\theta \otimes \nu \not\cong \theta$.
	
	We distinguish between orthogonal and symplectic representations via the Frobenius--Schur indicator: if $\pi$ is an irreducible self-dual representation of a finite group $G$ then $S(\pi):=\frac{1}{|G|} \sum\limits_{g \in G} \Tr \pi(g^2)$ is such that $S(\pi)=1$ if $\pi$ is orthogonal and $S(\pi)=-1$ if $\pi$ is symplectic. We show that $S((\theta \otimes \nu) \oplus \theta)=0$ from which the result follows.
	
	Since $\theta \otimes \nu^r$ only depends on $r \bmod{2}$ (as $\theta \otimes \nu^2 \cong \theta$), we have that $$fS((\theta \otimes \nu) \oplus \theta)= S \left(\theta \otimes \left( \bigoplus\limits_{r=1}^{2f} \nu^r \right) \right),$$ so we work with the right hand side since $\bigoplus\limits_{r=1}^{2f} \nu^r= \theta_{reg}$ is the inflation of the regular representation on $\Gal(L^I/K)$.
	
	Now $\Tr \theta_{reg} (\Frob^k \iota^l)= \Tr \theta_{reg} (\Frob^k) = 0$ unless $k=0$ so the only group elements that contribute are in the abelian subgroup $\langle \Frob^f, \iota \rangle$, where $\theta_{reg}$ acts as the identity.
	
	We now compute $$\frac{|\Gal(L/K)|}{\dim \theta_{reg}} S(\theta \otimes \theta_{reg}) = 2\sum\limits_{l=1}^e \Tr \theta (\iota^{2l});$$ the calculation of $S(\theta \otimes \theta_{reg})$ now follows.

	Finally observe that by Theorem \ref{props}iii, $$W(\theta \otimes \nu)=W(\theta)\nu(\Frob)^{a(\theta)} = W(\theta)\nu(\Frob)^{\dim \theta}=-W(\theta)$$ since $\nu$ is an unramified character of order $2\dim \theta$.
\end{proof}

Before we prove that we may suppose that the $\rho_e$ are symplectic, we need a lemma concerning the structure of semisimple symplectic representations.

\begin{lemma}
\label{semisymp}
	Let $\rho$ be a semisimple symplectic representation of a group $G$. Then there exists irreducible symplectic representations $\lambda_1, \cdots \lambda_t$ of $G$ and a representation $\pi$ of $G$ such that $$\rho \cong \pi \oplus \pi^* \oplus \lambda_1 \oplus \cdots \oplus \lambda_t.$$
\end{lemma}

\begin{proof}
	See \cite[Lemma A.2]{Sab07}.
\end{proof}

\begin{remark}
	Unfortunately, we cannot simply reorganise the summands between the different $\rho_e$ to show they are symplectic; indeed suppose $\rho_e$ is self-dual but not symplectic and consider the symplectic representation $\rho= \rho_e^{\oplus 2}$. This does not have a symplectic decomposition as desired due to our assumption of $\rho_e$ on inertia.
\end{remark}

We are finally in a position to show that the $\rho_e$ may be assumed to be symplectic, without affecting the root number of the overarching symplectic representation. This theorem follows from the previous two lemmas.

\begin{theorem}
\label{rhoesym}
	Let $\rho_e$ be a tamely ramified Weil representation as in Definition \ref{rhoe}. Write $$\rho_e \cong \pi \oplus \pi^* \oplus \bigoplus_{j=1}^m \theta_j,$$ such that $\theta_j$ are irreducible self-dual summands. Let $$\rho_e'=\pi \oplus \pi^* \oplus \bigoplus_{j=1}^m \theta_j \otimes \nu_j,$$ where $\nu_j$ are unramified characters of finite order such that $\theta_j \otimes \nu_j$ is symplectic for all $j$. Then $\rho_e'$ satisfies Definition \ref{rhoe} and $W(\rho_e')=(-1)^aW(\rho_e)$, where $a$ is the number of summands $\theta_j$ which are orthogonal.

Moreover, if $\rho$ is a symplectic tamely ramified Weil representation such that we have a decomposition $\rho \cong \bigoplus \rho_{e_j}$ into summands satisfying Definition \ref{rhoe}, then $W(\rho)=\prod W(\rho_{e_j}')$.
\end{theorem}

\begin{remark}
\label{rmkrhoe}
	If $B/K$ is an abelian variety with potentially good reduction, then $\rho_B \otimes \xcyc^{1/2}$ satisfies the conditions of Definition \ref{rhoe}, hence $\rho_B \otimes \xcyc^{1/2}$ decomposes into self-dual summands $\rho_{e_j}$ (cf. Definition \ref{rhoe}). By Corollary \ref{dualrep} and Theorem \ref{rhoesym}, we therefore have $$W(\rho_B)=W(\rho_B \otimes \xcyc^{1/2})=\prod\limits_{j=1}^k W(\rho_{e_j}'),$$ where $\rho_{e_j}'$ is the symplectic version of $\rho_{e_j}$ as in Theorem \ref{rhoesym}. Moreover, by applying the same ideas, we may suppose that if one irreducible summand $\theta$ (and hence all summands) of $\rho_e$ is self-dual, then $\theta$ is symplectic.
\end{remark}

All that remains is to explicitly describe $W(\rho_{e_j})$ when $\rho_{e_j}$ is symplectic.

\subsection{Jacobi symbols}

With a symplectic representation of the form $\rho_e$ fixed (cf. Definition \ref{rhoe}), we shall allow the prime to vary and study the effect on the root number of $\rho_e$. By Remark \ref{rmkrhoe}, we may assume that all irreducible self-dual summands are also symplectic and hence only consider that setting when applying Theorem \ref{selfdual}.

Under this symplectic assumption on an irreducible representation $\theta$, we further note that if $e$ is odd then $W(\theta)=-1$ independently of whether $q$ is odd or even; coupling this with Remark \ref{dualodde}, we see that we no longer need to distinguish between $q$ being odd or even in the setting of potentially good reduction.

Note that since we have only assumed tame inertia, $q$ is necessarily coprime to $e$ hence it is reasonable to expect a Jacobi symbol involving $q$ and $e$. Our first step in this direction is the following lemma.

\begin{lemma}
	Let $l$ be an odd prime and $e=l^k$. Then $W(\rho_e)=\left( \dfrac{q}{l} \right)$.
\end{lemma}

\begin{proof}
	Let $\xi$ be a generator of the cyclic group $(\Z/e\Z)^{\times}$ and write $q \equiv \xi^n \mod{e}$ with $n$ minimal. Note that the dimension of an irreducible summand is equal to the order of $q^n$, and there are $\gcd(n,\tilde{\varphi}(e))$ such summands.
		
	If $n$ is odd, then there are necessarily an odd number of irreducible summands which must all be self-dual. Note that if $e$ is odd, then the root number attached to each summand is necessarily negative and hence $W(\rho_e)=-1$.
		
	Now assume $n$ is even and observe that there are an even number of irreducible summands. If they are self-dual then their product must equal $1$ so we are done. Otherwise we have dual pairs whose root number must be trivial by Theorem \ref{dual} as $e$ is odd.
\end{proof}
	
We now consider the case $l=2$ to obtain a similar result. We note that cases $e=2,4$ have been covered by Rohrlich and since these give separate Legendre symbols, we shall omit the proofs here.
	
\begin{lemma}
	Let $e=2^n$ for $n \geqslant 3$. Then $W(\rho_e)=-1$ if and only if $q$ has order $2^{n-2} \mod{e}$ if and only if $\left( \dfrac{2}{q} \right)= -1$.	
\end{lemma}	
		
\begin{proof}
	By our hypotheses, $(\Z/e\Z)^{\times} \cong \Z/2\Z \times \Z/2^{n-2}\Z$. As $-1$ is not a square in $\Z_2$, the only self-dual case that occurs is when $q \equiv -1 \! \mod{e}$. In this case, we have $2^{n-2} \geqslant 2$ self-dual summands hence the root number is positive.
		
	We are now reduced to considering dual pairs. Let $f$ be the dimension of an irreducible summand, which divides $2^{n-2}$. If $f$ is not maximal, then we would have an even multiplicity of dual pairs so the root number is again $1$, so we may suppose $f=2^{n-2}$.
		
	Now there are 3 non-trivial square roots of $1$ and since the representation is not self-dual, we necessarily have $p^{f/2} \equiv 2^{n-1} \pm 1 \mod{2^n}$. Squaring this, we observe that $v_2(q^f-1)=n$.
		
	Now	$v_2\left(\gcd\left(2^n,\dfrac{q^f-1}{q-1}\right)\right) = n-v_2(q-1)$,	so the congruence condition we obtain from Lemma \ref{dualmod} is therefore $q \equiv 1 \mod{2^{v_2(q-1)}}$. Hence it will always produce a negative root number. The Jacobi symbol then follows by a direct computation.
\end{proof}

\begin{lemma}
	Let $l$ be an odd prime such that $l \equiv 3 \mod{4}$ and let $e=2l^k$. Then $W(\rho_e)=\left( \dfrac{-1}{q} \right)$.
\end{lemma}

\begin{proof}
	Let $f$ be the order of $q \bmod{e}$. Then an irreducible summand is self-dual if and only if $f$ is even; this implies we either have an odd number of dual pairs or an odd number of self-dual summands and hence $W(\rho_e)$ is the same as the root number of a dual pair or irreducible self-dual summand.
	
	First suppose $f$ is odd. Then $v_2(q-1)=v_2(q^f-1)$ and we find that by Theorem \ref{dual}, $W(\rho_e)=1 \Leftrightarrow v_2(q-1) \geqslant 2$, which agrees with the Jacobi symbol. Otherwise $f \equiv 2 \bmod{4}$ and we compute that $v_2(q^{f/2}+1)=v_2(q+1)$; the result now follows by Theorem \ref{selfdual}.
\end{proof}

We now cover the remaining cases and show that the root number is always positive.

\begin{lemma}
	\begin{enumerate}
		\item[]
		\item Let $e$ be an integer which is neither a prime power nor twice a prime power. Then $W(\rho_e)=1$ for all $q$.
		\item Let $l$ be an odd prime such that $l \equiv 1 \mod{4}$ and let $e=2l^k$. Then $W(\rho_e)=1$ for all $q$.
	\end{enumerate}
\end{lemma}
	
\begin{proof}
	(i): Observe that $(\Z/2\Z)^2 \leqslant (\Z/e\Z)^{\times}$ and hence in the self-dual case there are necessarily an even number of summands so we may reduce to dual pairs. Similarly, if $(\Z/2\Z)^3 \leqslant (\Z/e\Z)^{\times}$ then there are an even number of dual pairs so we are done. This reduces us to the case where $e=4l^k$ for some odd prime $l$.
	
	Let $f$ be the order of $q \bmod{e}$ and note that $v_2(f)<v_2(|(\Z/e\Z)^{\times}|) =1+v_2(l-1)$ since $(\Z/2\Z)^2 \leqslant (\Z/e\Z)^{\times}$. If $v_2(f) \neq v_2(l-1)$, then we again must have an even number of dual pairs and hence we may assume that $v_2(f)=v_2(l-1) \geqslant 1$.
	
	By Theorem \ref{dual}, we have $$W(\sigma \oplus \sigma^*)=1 \Leftrightarrow  v_2(q-1) \geqslant 3- \min \{ 2, v_2(q^{f-1}+q^{f-2}+ \cdots + 1) \}; $$ it is now a routine verification to check that the right hand side always holds.
	
	(ii): First note that $-1$ is a square $\!\!\! \mod{e}$ and that an irreducible summand is self-dual if and only if $f$, the order of $q \bmod{e}$, is even (and we have $q^{f/2} \equiv -1 \bmod{e}$). If it is not self-dual or $v_2(f) \neq v_2(\tilde{\varphi}(e))$, then we have an even number of dual pairs or self-dual summands so are done. 
	
	Hence $4|f$ and Theorem \ref{selfdual} states that we have a positive root number if and only if $v_2(q^{f/2}+1)=1$; this is true for all odd $q$.
\end{proof}

\begin{theorem}
\label{goodrootpf}
Let $e$ be a positive integer and let $\rho_e$ be a symplectic Weil representation as defined in Definition \ref{rhoe}. Let $q$ be the cardinality of the residue field of $K$. Then 
	
$W(\rho_e) = \begin{cases}
	\left( \dfrac{q}{l} \right) &\text{ if } e=l^k \quad \text{ for some odd prime $l$ and integer $k\geqslant 1$}; \\[10pt]	
	\left( \dfrac{-1}{q} \right) &\text{ if } e=2l^k \quad \text{ for some prime } l \equiv 3 \bmod{4}, k\geqslant 1,  \qquad \text{ or } e=2; \\[10pt]	
	\left( \dfrac{-2}{q} \right) &\text{ if } e=4; \\[10pt]	
	\left( \dfrac{2}{q} \right) &\text{ if } e=2^k \quad \text{ for some integer } k \geqslant 3; \\[10pt]	
	\quad 1 &\text{ else.}
\end{cases}$
\end{theorem}

\begin{proof}
	This is the amalgamation of the previous four lemmas in this section.
\end{proof}

\newpage
\section{Table of lawful genus two hyperelliptic curves}
\label{lawfultable}
	Below we give a table of lawful genus two hyperelliptic curves over $\Q$, ordered by conductor up to $50,000$. The list of curves used as our input data was obtained from \cite{BSSVY16,LMFDB}. As root numbers are invariant under isogeny, we only list curves with non-isogenous Jacobians. \\
	
\begin{tabular}{l|c|c}
	\multicolumn{1}{c|}{$f(x)$}	& Conductor & Lawful good/ evil \\ \hline
	\rule{0pt}{2.5ex}$x^6+4x^5+6x^4+2x^3+x^2+2x+1$ & $169$ & Good \\
	$x^6-4x^5+2x^4+2x^3+x^2+2x+1$ & $529$ & Good \\
	$x^6+2x^5+7x^4+6x^3+13x^2+4x+8$ & $841$ & Good \\
	$x^6+4x^5+6x^4+6x^3+x^2-2x-3$ & $961$ & Good \\
	$x^6-2x^4+6x^3+13x^2+6x+1$ & $3721$ & Good \\
	$x^6+4x^5+2x^4+2x^3+x^2-2x+1$ & $4489$ & Good \\
	$-3x^6+4x^5-2x^4+2x^3+x^2-2x+1$ & $4489$ & Evil \\
	$-3x^6+2x^5+29x^4-6x^3-82x^2+4x+73$ & $4489$ & Good \\
	$x^6+2x^5+x^4+6x^3+2x^2-4x+1$ & $5329$ & Good \\
	$-3x^6-32x^5-62x^4+102x^3-159x^2+126x-31$ & $5329$ & Good \\
	$4x^5+5x^4+6x^3-3x^2-8x-4$ & $5929$ & Good \\
	$x^6-12x^5+38x^4-26x^3-7x^2+6x+1$ & $8281$ & Good \\
	$4x^5+33x^4+46x^3+13x^2-4x$ & $8281$ & Good \\
	$x^6+2x^5+9x^4+10x^3+26x^2+12x+25$ & $9409$ & Good \\
	$x^6+2x^4+2x^3+5x^2+6x+1$ & $10609$ & Good \\
	$x^6-10x^4+2x^3+21x^2-18x+5$ & $10609$ & Evil \\
	$x^6+2x^5+5x^4+2x^3-2x^2-4x-3$ & $11449$ & Good \\
	$4x^5-11x^4+2x^3+9x^2-4x$ & $11881$ & Good \\
	$-3x^6-4x^5+30x^4+30x^3-111x^2-50x+137$ & $17689$ & Good \\
	$4x^5-15x^4+10x^3+5x^2-4x$ & $17689$ & Good \\
	$x^6-10x^4-10x^3+5x^2+6x+1$ & $17689$ & Good \\
	$x^6+8x^5+10x^4+6x^3+5x^2+2x+1$ & $17689$ & Good \\
	$x^6+2x^5+9x^4+2x^3-6x^2-28x+21$ & $17689$ & Good \\
	$4x^5+17x^4+14x^3-3x^2-4x$ & $24649$ & Good \\
	$x^6+4x^5+2x^4+2x^3-3x^2-2x-3$ & $27889$ & Good \\
	$x^6-8x^4-8x^3+8x^2+12x-8$ & $28561$ & Evil \\
	$x^6+4x^5+2x^4+2x^3+41x^2+78x+41$ & $32761$ & Good \\
	$x^6+2x^4+2x^3+5x^2-6x+1$ & $36481$ & Good \\
	$x^6+2x^4-14x^3+5x^2+6x+1$ & $37249$ & Good \\
	$x^6+2x^5+3x^4+4x^3+7x^2+14x+13$ & $43681$ & Good \\
	$-3x^6+2x^5+21x^4-18x^3-30x^2+16x+17$ & $44521$ & Good \\
	$x^6+4x^5+2x^4+6x^3+x^2-2x+1$ & $48841$ & Good \\
	$-3x^6+8x^5-18x^4+26x^3-23x^2+10x-3$ & $48841$ & Good \\
	$x^6+4x^5-4x^4-22x^3+8x^2+8x-71$ & $49729$ & Good
\end{tabular}

\begin{remark}
	Note that for conductors $4489$ and $10609$, we get examples of both good and evil lawful Jacobians; in the former case the lawful good Jacobians also have distinct analytic ranks.
	
	On the other hand, we have only lawful good Jacobians for conductors $5329, 8281, 17689$ and $48841$. Moreover, every isogeny class of Jacobians of conductor $5329, 17689$ and $48841$ listed in \cite{BSSVY16,LMFDB} is lawful.
\end{remark}

\bibliographystyle{alpha}
\bibliography{Everything}

\end{document}